\documentclass{amsart}

\usepackage{amsmath,latexsym,amscd,amsbsy,amssymb,amsfonts,amsthm,fleqn,leqno,
euscript, graphicx, color}

\newcommand{\Hom}{\mbox{{\rm Hom}}}

\newtheorem{thm}{Theorem}[section]
\newtheorem{pro}[thm]{Proposition}
\newtheorem{lem}[thm]{Lemma}
\newtheorem{cor}[thm]{Corollary}
\newtheorem{ex}[thm]{Example}

\newtheorem{df}[thm]{Definition}
 
\newtheorem{rem}[thm]{Remark}

\newcommand{\cl}{\mbox{{\rm cl}}}
\newcommand{\st}{\mbox{{\rm St}}}
\newcommand{\pr}{\mbox{{\rm pr}}}
\newcommand{\id}{\mbox{{\rm id}}}
\newcommand{\aut}{\mbox{{\rm Aut}}}

\begin{document}

\title{Enveloping Ellis semigroups as compactifications of  transformations groups}

\author{K.\,L.~Kozlov}\thanks{The study of the first author was carried out with the help of the Center of Integration in Science, Ministry of Aliyah and Integration, Israel,  8123461 and is supported by ISF grant 3187/24.}
\address{Department of Mathematics, Ben-Gurion University of the Negev, Beer Sheva, Israel}
\email{kozlovk@bgu.ac.il}
\author{B.\,V.~Sorin}
\address{Lomonosov Moscow State University}
\email{bvs@improfi.ru}


\date{}

 \maketitle
 
\begin{abstract} The notion of a proper Ellis semigroup compactification is introduced. Ellis's functional approach shows how to obtain them from totally bounded  equiuniformities on a phase space $X$ when the acting group $G$ is with the topology of pointwise convergence and the $G$-space $(G, X, \curvearrowright)$ is $G$-Tychonoff. 

The correspondence between proper Ellis semigroup compactifications of a topological group and special totally bounded equiuniformities (called Ellis equiuniformities) on a topological group is established. The Ellis equiuniformity on a topological transformation group $G$ from the maximal equiuniformity on a phase space $G/H$ in the case of its uniformly equicontinuous action is compared with Roelcke uniformity on $G$. 

Proper Ellis semigroup compactifications are described for groups ${\rm S}(X)$ (the permutation group of a discrete space $X$) and $\aut (X)$ (automorphism group of an ultrahomogeneous chain $X$) in the permutation topology. It is shown that this approach can be applied to the unitary group of a Hilbert space. 

\medskip
 

\medskip

 \end{abstract}

\section{Introduction and preliminary remarks} 

Compactifications of topological (semi)groups are a vital tool of investigations in Topological dynamics~\cite{Vries}, \cite{GlasnerMegr} and Topological Algebra~\cite{HS}. In the paper by a proper compactification of a topological space $X$, we understand a compactum in which $X$ is a dense subspace. Due to A.~Weil a topological group $G$ is a dense subgroup of a compact topological group $K$ ($K$ is a proper compactification of $G$) iff $G$ is totally bounded (in the right uniformity). There are groups without (proper) semitopological semigroup compactifications~\cite{Megr2001}. Every topological group has a proper right topological semigroup compactification, the greatest ambit~\cite{Brook}. A self-contained treatment of the theory of compact right topological semigroups and, in particular, of semigroup compactifications (nonproper) is in~\cite{BergJungMiln}. A study of the proper group's compactifications can be found  in~\cite{8}.
 
The functional approach in the study of transformation groups (groups acting effectively on phase spaces) allows us to view elements of $G$ as self maps of the phase space $X$. Each $g\in G$ is an element of $X^X$ and the composition of maps makes the image of $G$ a semigroup. Moreover, the multiplication on $G$ and composition on $X^X$ makes this representation an isomorphism of semigroups. This approach came from topological dynamics due to Ellis's construction~\cite{Ellis}. R.~Ellis examined the subgroup $G$ of group of homeomorphisms $\Hom(X)$ of compactum $X$ as a subsemigroup of the Cartesian product $X^X$ (a right topological semigroup) and named the closure of $G$ in $X^X$ the {\it enveloping {\rm(}Ellis{\rm)} semigroup}~\cite{Ellis}. It is a compact right topological semigroup and is a nonproper semigroup compactification of $G$ ($G$ is not its topological subspace). 

\medskip

However, if $G$ is in the topology of pointwise convergence (for the action $G\curvearrowright X$), then Ellis's representation becomes a topological isomorphism. For the action of a topological group on itself by multiplication on the left the topology of a group is equal to the topology of pointwise convergence (Lemma~\ref{cointpc}). If the topology of pointwise convergence on a group is an admissible group topology for the action $G\curvearrowright X$ and $(G, X, \curvearrowright)$ is a $G$-Tychonoff space, then the topology of pointwise convergence on a group is an admissible group topology for the action of $G$ of any 
$G$-compactification of $X$ (Lemma~\ref{l1}). These observations allow us to define in \S~\ref{Defcomp} a concept of a proper Ellis semigroup compactification (Definition~\ref{DEF}) and the corresponding Ellis equiuniformity (Definition~\ref{def1}). The proper Ellis semigroup compactification is a right topological semigroup and is a $G$-compactification of a group $G$ (multiplication agrees with the extended action). We demonstrate how to obtain all Ellis  equiuniformities from totally bounded equiuniformities on $G$ ($G$ is examined as a $G$-Tychonoff space $(G, G, \cdot)$) in \S~\ref{descrcomp} (Theorem~\ref{mathfrak}). 

In \S~\ref{unifequi} we examine uniformly equicontinuous actions (in this case the topology of pointwise convergence is an admissible group topology on an  acting group, and actions by isometries and actions on coset spaces with respect to neutral subgroups (by multiplication on the left) are among them). For actions on coset spaces with respect to a neutral subgroup, the corresponding Ellis equiuniformity is less than or equal to the least Ellis equiuniformity on a group that is greater than or equal to the precompact reflection of the Roelcke uniformity (Theorem~\ref{autultr}). 

Examples illustrate what proper Ellis semigroup compactifications can be obtained from the totally bounded maximal equiuniformities on phase spaces.

In \S~\ref{discrete} we examine an example of an ultratransitive action of a group $G$ on a discrete space $X$. It is shown that the Roelcke compactification of $G$ is equal to the proper Ellis semigroup compactification obtained from the maximal equiuniformity on $X$ and is a semitopological semigroup. The maximal equiuniformity in this case is the least totally bounded uniformity on $X$ (and corresponds to the Alexandroff compactification of $X$). 
The Roelcke compactification of $G$ is described (Theorem~\ref{Roelcke precomp4-1} and Corollary~\ref{cordis}).  Houghton's groups serve as an example.  In~\cite[\S\ 12, Theorem 12.2]{GlasnerMegr2008} the description of the Roelcke compactification of the permutation group of natural numbers is given. It is also shown that it is a semitopological group homeomorphic to the Cantor set.

In \S~\ref{chain} we examine an ultratransitive action of a subgroup $G$ of automorphism group $\aut (X)$ on a chain $X$ in discrete topology. We establish that the Roelcke compactification of $G$ is a proper Ellis semigroup compactification iff a chain is continuous (Corollaries~\ref{contchain} and~\ref{ccc}). In the general case, we derive the least Ellis equiuniformity which is greater than or equal to the Roelcke uniformity, from the Roelcke uniformity using Ellis construction (Theorem~\ref{Roelcke precomp4-2-1} and Corollary~\ref{ccc}). The proper Ellis semigroup compactification and the  Roelcke compactification of $G$ are described in Theorem~\ref{Roelcke precomp4-2-1}. The maximal equiuniformity on $X$ from which the Ellis equiuniformity is obtained corresponds to the maximal $G$-compactification of $X$ which is also described in\ \S~\ref{chain}. It is the minimal linearly ordered compactification of a GO-space $X$. 
Thompson group $F$ in the permutation topology is examined in Example~\ref{Thompson}. $F$ is Roelcke precompact but its Roelcke compactification is not a proper Ellis semigroup compactification. 

In \S~\ref{UNIT} it is shown that the Roelcke uniformity on the unitary group of a Hilbert space coincides with the uniformity from the family of the stabilizers of points of the unit sphere and is an Ellis equiuniformity  (Theorem~\ref{Roelcke precompHilb}). The result provides an alternative proof that the Roelcke compactification of the unitary group is a semitopological group which coincides with the space of all linear operators of norm $\leq 1$ in the weak operator topology~{\rm\cite{U1998}}.

\medskip

The method for working with Ellis equiuniformities is detailed in \S~\ref{digonal} and~\ref{Elliscmp}. In Theorem~\ref{diagactalgcont} we present both the topological and uniform variants of Ellis's construction and establish the extremal property of Ellis uniformity. Proposition~\ref{order} outlines a condition for the coincidence of the Ellis equiuniformity corresponding to the extended action and equiuniformity  corresponding to the original one. We demonstrate that the map of the poset of equiuniformities on the phase space to the poset of Ellis equiuniformities preserves order. In Proposition~\ref{a2-2} we establish a condition when the Ellis equiuniformity is equal to the uniformity on a group from a family of small subgroups.  Lastly, in Proposition~\ref{connection} we establish the connection between Ellis equiuniformities for the actions on different phase spaces.

The paper is a continuation of investigations initiated in~\cite{Sorin1}.

\medskip

We adopt terminology and notation from~\cite{Engelking} and~\cite{RD}. The family of finite subsets of a set (or space) $X$ is denoted $\Sigma_X$ and directed by inclusion. The set of non-negative reals is represented by  $\mathbb R_+$.

All spaces considered are Tychonoff, and we denote a (topological) space as $(X, \tau)$ where $\tau$ is a topology on (a set) $X$.  We order topologies: $\tau'\geq\tau$ if $\tau\subset\tau'$. An abbreviation ``nbd''  refers to an open neighbourhood of a point. The family of all nbds of the unit $e$ in $G$ is denoted $N_{G}(e)$. A (closed) subgroup $H$ of a topological group $G$ is called {\it neutral} if for any $O\in N_{G}(e)$ there exists $V\in N_{G}(e)$ such that $VH\subset HO$~\cite{RD}. A proper compactification of $X$ is denoted $(bX, b)$ where $b X$ is compacta, $b: X\to b X$ is a dense embedding.

Uniform space is denoted as $(X, \mathcal U)$, $u$ is a cover and ${\rm U}$ is a corresponding entourage of $\mathcal U$. If the cover $v$ is a refinement of  $u$, we use the notation $v\succ u$. Uniformity on a topological space is compatible with its topology. Additionally, $(\tilde X, \tilde{\mathcal U})$ denotes the completion of $(X, \mathcal U)$.


\subsection{Semigroups}\label{semgr}
A {\it semigroup} $(S, \bullet)$ is a set $S$ with an associative internal binary operation (multiplication $\bullet$). A group $G$ with multiplication $\cdot$ is a semigroup with multiplication  $\bullet=\cdot$.  A map $s: (S, \bullet)\to (S', \bullet')$ of semigroups is a {\it homomorphism} if $s(x\bullet y)=s(x)\bullet' s(y)$ for all $x, y\in S$. If $s: (G, \cdot)\to (S, \bullet)$ is a homomorphism, then $s(G)$ is a group and a subsemigroup of $S$. 

A  {\it semitopological} ({\it right topological}) {\it semigroup} is a semigroup $(S, \bullet)$ on a topological space $S$ with separately continuous multiplication  $\bullet$ (multiplication  $\bullet$ continuous on the right). 

All the necessary (and additional) information can be found in~\cite{BergJungMiln} and~\cite{ArhTk}.


\subsection{Spaces of maps and semigroup structure on $X^X$}\label{funcsp}

1. Since the Cartesian product $X^X$ may be regarded as the set of maps of $X$ into $X$ ($f=(f(x))_{x\in X}\in X^X$), a semigroup structure (multiplication is composition $\circ$ of maps) may be introduced on $X^X$.

If $X$ is a topological space, then the multiplication on $X^X$ in the Tychonoff topology (the subbase is formed by the sets $[x, O]=\{f\in X^X\ |\ f(x)\in O\}$, $x\in X$, $O$ open subset of $X$) is continuous on the right ($f\to f\circ g$; if $f\in [g(x), O]$, then $f\circ g\in [x, O]$) and continuous on the left for continuous maps  ($f\to g\circ f$; $g\in C(X)$, if $f\in [x, g^{-1}O]$, then $g\circ f\in [x, O]$).  Thus, $(X^X, \circ)$ is a {\it right topological semigroup}.

If $(X, \mathcal U_X)$ is a uniform space, then the uniformity $\mathcal U_X^{p}$ on $X^X$ is the Cartesian product of uniformities  $\mathcal U_X$ on factors  (see~\cite[Ch.~8, \S\ 8.2]{Engelking}). The multiplication on $(X^X, \mathcal U_X^{p})$ is uniformly continuous on the right ($f\to f\circ g$; if $y=g(x)$ and $(f_1 (y), f_2 (y))\in U\in\mathcal U_X$, then $((f_1\circ g)(x), (f_2\circ g)(x))\in U$) and uniformly continuous on the left for uniformly continuous maps  ($f\to g\circ f$, $g$ is uniformly continuous; for ${\rm U}\in\mathcal U_X$ let  ${\rm V}\in\mathcal U_X$ be such that if $(x, y)\in {\rm V}$, then $(g(x), g(y))\in {\rm U}$;  thus if $(f_1 (x), f_2 (x))\in {\rm V}$, then $((g\circ f_1)(x), (g\circ f_2)(x))\in {\rm U}$).

The topology on  $X^X$ induced by  $\mathcal U_X^{p}$ is the Tychonoff topology of the product where every factor has the Tychonoff topology induced by  $\mathcal U_X$. Following~\cite[Ch.10 , \S\ 1]{Burb}  the uniformity $\mathcal U_X^{p}$ on $X^X$ (as the set of maps) is called {\it the uniformity of pointwise} (or {\it simple}) {\it convergence} (abbreviation u.p.c) and the induced topology is {\it the topology of pointwise convergence} $\tau_p$ (abbreviation t.p.c.)~\cite[Ch.\ 2, \S\ 2.6]{Engelking}. 

\medskip

2. If $(X, \mathcal U_X)$ is a uniform space, then, following~\cite[Ch.10 , \S\ 1]{Burb} (see, also~\cite[Ch.~8, \S\ 8.2]{Engelking}), the {\it uniformity of uniform convergence} $\mathcal U_X^{u}$ on $X^X$ (abbreviation u.u.c.) is defined.  The topology  $\tau_u$ on  $X^X$ induced by  $\mathcal U_X^{u}$ is the {\it topology of uniform convergence}  (abbreviation t.u.c.). Evidently, ${\mathcal U}_X^{p}\subset\mathcal U_X^u$ and $\tau_u\geq\tau_p$.


\subsection{$G$-spaces, admissible group topologies on a transformation group and the topology of pointwise convergence}\label{admtop}

\quad\ 1. An action $\alpha: G\times X\to X$ (designation $\alpha (g, x)=gx$) of a group $G$ on a set $X$ ({\it phase space}) is called {\it effective} if the {\it kernel of the action} $\{g\in G\ |\ gx=x,\ \forall\ x\in X\}$ is the unit of $G$. If $G$ effectively acts on $X$, then $G\subset {\rm S}(X)$, where ${\rm S}(X)$ is the {\it permutation group} of $X$. 

The subgroup $\st_{\sigma}=\{g\in G\ |\ gx=x,\ x\in\sigma\in\Sigma_X\}=\bigcap\limits_{x\in\sigma}\st_{x}$ of $G$ is a {\it stabilizer}  (or stabilizer of points from $\sigma$).

If $T\subset G$, $Y\subset X$, then $TY=\alpha(T, Y)=\bigcup\{ ty\ |\ t\in T,\ y\in Y\}$. For $x\in X$ $\alpha_x: G\to X$, $\alpha_x(g)=\alpha (g, x)$, is an {\it orbit map}. 

\medskip

2. Under a continuous action $\alpha: G\times X\to X$ of a topological group $G$ on a space $X$, the triple $(G, X, \alpha)$ is called a {\it $G$-space}. A  continuous action $\alpha: G\times X\to X$ will be denoted $G\curvearrowright X$ if we don't need the special designation for the map. $G$-spaces $(G, X, \alpha)$ and  $(G, X', \alpha')$ are {\it equivalent} if there exists a homeomorphism $\varphi: X\to X'$ such that $\varphi(\alpha (g, x))=\alpha'(g, \varphi(x))$.  Equivalent $G$-spaces are identified.

A uniformity $\mathcal U_X$ on $X$ is called an {\it equiuniformity}~\cite{Megr1984} if the action $G\curvearrowright X$ is {\it saturated} (any homeomorphism from $G$ is uniformly continuous) and is {\it bounded} (for any $u\in\mathcal U$ there are $O\in N_G(e)$ and $v\in\mathcal U$ such that the cover $Ov=\{OV\ |\ V\in v\}\succ u$). In this case  $(G, X, \alpha)$ a {\it $G$-Tychonoff space}, the action is extended to the continuous action  $\tilde\alpha: G\times \tilde X\to\tilde X$ on the completion $(\tilde X, \tilde{\mathcal U}_X)$ of $(X, \mathcal U_X)$. The extension $\tilde{\mathcal U}_X$ of $\mathcal U_X$ to $\tilde X$ is an equiuniformity on $\tilde X$. 

\noindent The embedding $\jmath: X\to\tilde X$ is a {\it $G$-map} or an {\it equivariant map} of $G$-spaces, i.e. that the following diagram 
$$\begin{array}{ccc}
\quad G\times X & \stackrel{{\id\times\jmath}} {\longrightarrow} & G\times \tilde X \\
\alpha  \downarrow &   & \quad \downarrow \tilde\alpha \\
\quad X & \stackrel{\jmath} {\longrightarrow} & \tilde X, 
\end{array}$$  
is commutative, the pair $(\tilde X, \jmath)$  is a  {\it $G$-extension of $X$} and $\jmath (X)$ is a dense {\it invariant} subset of $\tilde X$~\cite{Megr0}. $G$-spaces as extensions will also be denoted as $(G, (\tilde X, \jmath), \tilde\alpha)$.

If $\mathcal U_X$ is  a totally bounded equiuniformity on $X$, then $(b X=\tilde X, \jmath_b)$ is a  {\it $G$-compactification}  or an {\it equivariant compactification} of  $X$. If  $(G, X, \curvearrowright)$ is a $G$-Tychonoff space,  then the maximal (totally bounded) equiuniformity on $X$ exists and {\it the maximal $G$-compactification}  $(\beta_G X, \jmath_{\beta})$ of $X$ corresponds to the maximal totally bounded  equiuniformity on $X$. The maximal totally bounded equiuniformity  $\mathcal U_X^{bm}$ on $X$ is the {\it precompact reflection}~\cite{Isbell} of the maximal equiuniformity  $\mathcal U_X^{m}$ on $X$ and the compactification  $(b X, \tilde{\mathcal U}_X^{bm})$ (completion of $(X, {\mathcal U}_X^{bm})$) is the {\it Samuel compactification of $X$ with respect to the equiuniformity $\mathcal U_X^{m}$} (see~\cite[Ch.\ 8, Problem 8.5.7]{Engelking}).

\begin{rem} {\rm If  for an action $G\curvearrowright X$ (not continuous) of a topological group $G$ on a space $X$ the uniformity $\mathcal U_X$ on $X$ is saturated and  bounded, then the the action $G\curvearrowright X$ is continuous.}
\end{rem}

The maximal equiuniformity on a coset space $G/H$ of a topological group $G$ (action by multiplication on the left) is described in~\cite{ChK} or~\cite{ChK2}. There such actions are called {\it open} and the covers of the base of the maximal equiuniformity are of the form 
$$\{Ox\ |\ x\in G/H\},\ O\in N_G(e).$$

We denote by $\mathbb{BU}(X)$ the set of all totally bounded equiuniformities on $X$. All the necessary (and additional) information can be found in~\cite{Megr0}.

\medskip

3. If $X$ is a topological space, $\Hom(X)$ is the group of its homeomorphisms, then a group $G$ {\it effectively acts on $X$} if $G\subset\Hom(X)$. If $X$ is a discrete space, then $\Hom(X)={\rm S}(X)$.

A topology in which a group $G$ is a topological group and its effective action $G\curvearrowright X$ is continuous is called an {\it admissible group topology on $G$}~\cite{Arens}. 

For an effective action of $G$ on a compactum $X$ the {\it compact-open topology} $\tau_{co}$ is the smallest admissible group topology on $G$~\cite{Arens}. Considering elements of $G$ as continuous maps of $X$ into $X$ ($G\subset X^X$) the compact-open topology $\tau_{co}$ coincides with the  t.u.c.  $\tau_{u}$~\cite[Corollary 8.2.7]{Engelking}.

For an effective action of $G$ on a discrete space $X$, the {\it permutation topology} $\tau_{\partial}$ (the subbase  is formed by the sets $[x, y]=\{g\in G\ |\ gx=y\}$, $x, y\in X$) is the smallest admissible group topology. The group $G$ in the permutation topology is {\it non-Archimedean} (a unit nbd base is formed by clopen subgroups (stabilizers)) and the $G$-space $((G, \tau_{\partial}), X, \curvearrowright)$ is $G$-Tychonoff.

If the {\it topology of pointwise convergence} $\tau_p$ (the subbase is formed by sets of the form $[x, O]=\{g\in G\ |\ gx\in O\}$, $x\in X$, $O$ is open in $X$) is an admissible group topology on a group $G$ of homeomorphisms of the space $X$ acting effectively, then it is the smallest admissible group topology \cite[Lemma 3.1]{Kozlov}, $\tau_{\partial}\geq\tau_p$, $\tau_{\partial}$ is an admissible group topology and  $\tau_{\partial}=\tau_p$, if $X$ is a discrete space. Considering elements of $G$ as continuous maps of $X$ into $X$ ($G\subset X^X$) the  topology of pointwise convergence $\tau_p$ coincides with the  t.p.c. from \S~\ref{funcsp}~\cite[Proposition 2.6.3]{Engelking}.

If the topology of pointwise convergence  $\tau_p$ is an admissible group topology on $G$ acting effectively on a compactum $X$, then $\tau_p=\tau_{co}$.

\begin{rem}{\rm 
Example from~{\rm\cite{Megr1988}} shows that the t.p.c. $\tau_p$ may be an admissible group topology on a group $G$ of homeomorphisms of $X$ acting effectively, but the $G$-space $(G=(G, \tau_p), X,  \curvearrowright)$ may not be $G$-Tychonoff.}
\end{rem}

In the paper, effective actions are examined.

\medskip

4. Let  $G$ be a topological group. The $G$-space $(G, G, \cdot)$  with the action $G \curvearrowright G$ by multiplication $\cdot$ on the left is $G$-Tychonoff.  

\begin{lem}\label{cointpc}
Let  $(G, \tau)$ be a topological group. For the $G$-space $(G, G, \cdot)$ the t.p.c. $\tau_p$ is an admissible group topology on $G$ and $\tau=\tau_p$. 
\end{lem}

\begin{proof}
The topology $\tau_p$ of pointwise convergence is an admissible group topology since the action is uniformly equicontinuous with respect to the left uniformity $L$ on $G$~\cite[Ch. X, § 3, item 5]{Burb} and $\tau\geq\tau_p$. If $O\in N_{(G, \tau)}(e)$, then the set $[e, O]\in\tau_p$ and $[e, O]=O$. Thus,  $\tau=\tau_p$.  
\end{proof}

\begin{lem}\label{l1} Let $(G=(G, \tau_p), X, \alpha)$ be $G$-Tychonoff and $\mathcal U_X$  is an equiuniformity on $X$. Then the t.p.c. $\tilde\tau_p$ on $G$  for the extended action $\tilde\alpha: G\times\tilde X\to\tilde X$ {\rm(}$(\tilde X, \tilde{\mathcal U}_X)$ is the completion of $(X, \mathcal U_X)${\rm)} is the smallest admissible group topology on $G$ and $\tilde\tau_p=\tau_p$.
\end{lem}

\begin{proof} Let $\mathcal V_X$ be the precompact reflection of  $\mathcal U_X$,  $(b X, \tilde{\mathcal V}_X)$ is the Samuel compactification of  $(X, \mathcal U_X)$ and 
$(b X, \jmath_{\mathcal V_X})$ is the $G$-compactification of $X$. From the correctly defined commutative diagram 
$$\begin{array}{ccl}
X & \stackrel{\jmath_{\mathcal V_X}}{\hookrightarrow} & \quad b X \\
\quad \jmath_{\mathcal U_X} \searrow &   &  \nearrow \jmath \\
 & \tilde X & 
\end{array}$$
where all maps are equivariant embeddings and $\jmath$ is the extension of a uniformly continuous map $\jmath_{\mathcal V_X}$ (see~\cite[Ch.\ 8, Problem 8.5.7]{Engelking}), it follows that $\jmath_{\mathcal V_X} (X)(=X)$ and $\jmath (\tilde X)(=\tilde X)$ are invariant subsets of $b X$.

The compact-open topology $\tau_{co}$ on $G$ induced by the extended action $G\curvearrowright b X$ is the smallest admissible group topology on $G$ and $\tau_p$ is an admissible group topology on $G$. Hence, $\tau_p\geq\tau_{co}$. Evidently, $\tau_{co}\geq\tau^s_p$, where $\tau^s_p$ is the t.p.c. on $G$ for the action $G\curvearrowright b X$ (not necessarily an admissible group topology).

Since $X$ and $\tilde X$ are invariant subsets of $b X$, for the t.p.c. $\tilde\tau_p$ on $G$ for the action $G\curvearrowright\tilde X$  (not necessarily an admissible group topology) one has $\tau^s_p\geq\tilde\tau_p\geq\tau_p$. Therefore, $\tau_{co}=\tilde\tau_p=\tau_p$ and $\tilde\tau_p$ is the smallest  admissible group topology for the extended action $G\curvearrowright\tilde X$.
\end{proof}


\subsection{Uniformities on a topological group}\label{uniftopgr}

On a topological group $G$ four uniformities are well-known. The {\it right uniformity} $R$ (the base is formed by the covers $\{Og=\bigcup\{h\cdot g\ |\ h\in O\}\ |\ g\in G\}$, $O\in N_G(e)$), the {\it left uniformity} $L$, the {\it two sided uniformity} $R\vee L$ (the least upper bound of the right and left uniformities) and the {\it Roelcke uniformity} $L\wedge R$ (the greatest lower bound of the right and left uniformities) (the base is formed by the covers $\{OgO=\bigcup\{h\cdot g\cdot h'\ |\ h, h'\in O\}\ |\ g\in G\}$, $O\in N_G(e)$). A group  $G$ is {\it Roelcke precompact} if the Roelcke uniformity is totally bounded. 

For every equiuniformity $\mathcal U$ on $G$ of a $G$-space $(G, G, \cdot)$ one has $\mathcal U\subset R$, $L\wedge R$  is an equiuniformity. All the necessary information about these uniformities can be found in~\cite{RD}. 

\medskip

{\it Uniformity on $G$ from a family of small subgroups is introduced in~{\rm\cite[\S\ 4]{Kozlov}}.}
Let  $(G=(G, \tau_p), X,  \curvearrowright)$ be a $G$-space.
The family of stabilizers  
$$\st_{\sigma}=\bigcap\{\st_{x}\ |\ x\in\sigma\},\ \sigma\in\Sigma_X,$$ 
is a {\it directed family  of small} (closed) {\it subgroups of $G$} ($\st_{\sigma}\leq\st_{\sigma'}\Longleftrightarrow \sigma\subset\sigma'$)  which defines the uniformity $R_{\Sigma_X}$ on $G$. The base of  $R_{\Sigma_X}$ is formed by the covers 
$$\{Og\st_{\sigma}=\bigcup\{h\cdot g\cdot h'\ |\ h\in O,\ h'\in\st_{\sigma}\}\ |\ g\in G\},\ O\in N_G(e),\  \sigma\in\Sigma_X~\mbox{\cite[\S\ 4]{Kozlov}}.$$  
Uniformity  $R_{\Sigma_X}$ is an equiuniformity (for the action $G \curvearrowright G$) and has a spectral representation. It is initial~\cite[Proposition-definition 0.16]{RD} with respect to the quotient maps $q_{\sigma}: G\to (G/\st_{\sigma}, \mathcal U_{G/\st_{\sigma}})$  (cosets $G/\st_{\sigma}$ with the maximal equiuniformities $\mathcal U_{G/\st_{\sigma}}$ (the quotient uniformities of the right uniformity $R$ on $G$)~\cite[Theorem 5.21]{RD} and~\cite{ChK}), $\sigma\in\Sigma_X$.
 
The completion of $G$ with respect to  $R_{\Sigma_X}$ is uniformly equivalent to the limit of the inverse spectrum $\{\widetilde{G/\st_{\sigma}}, p_{\sigma'\sigma}, \Sigma_X\}$, where $\widetilde{G/\st_{\sigma}}$ is the completion of $G/\st_{\sigma}$ with respect to $\mathcal U_{G/\st_{\sigma}}$,  $p_{\sigma'\sigma}: \widetilde{G/\st_{\sigma'}}\to\widetilde{G/\st_{\sigma}}$ is the extension of the quotient map $G/\st_{\sigma'}\to G/\st_{\sigma}$, $\sigma\subset\sigma'$~\cite[Theorem 1.4]{Kozlov}.

\begin{rem}{\rm\cite[Corollary 4.5]{Kozlov}}\label{a2-1}
{\rm $L\wedge R\subset R_{\Sigma_X}\subset R$, hence, if the uniformity $R_{\Sigma_X}$ is totally bounded, then $G$ is Roelcke precompact.

For a group $G$ {\rm(}acting on a discrete space $X${\rm)} in the permutation topology the stabilizers $\st_{\sigma}$, $\sigma\in\Sigma_X$, are clopen subgroups {\rm(}nbds of the unit{\rm)} and $L\wedge R=R_{\Sigma_X}$.}
\end{rem}


\section{Ellis uniformity on a group from an equiuniformity on a phase space}

\subsection{A proper enveloping Ellis semigroup compactification of a topological group}\label{Defcomp} 

\begin{df}\label{DEF}
Let $G$ be a topological group. A  proper enveloping Ellis semigroup compactification of $G$ is a pair $(S, s)$, where $S=(S, \bullet)$ is a compact right topological semigroup, $s: G\to S$ is a topological isomorphism of $G$ on a dense subgroup of $S$ and the restriction of the multiplication to $s(G)\times S$ is continuous. 

$\mathfrak{E} (G)$ is the set of proper enveloping Ellis semigroup compactifications of $G$.
\end{df}

\begin{rem} {\rm ( a) If a topological group $G$ is a subsemigroup of a compact semitopological semigroup $S$, then  the restriction of the multiplication to $G\times S$ is necessarily continuous~{\rm\cite[Proposition 6.2]{Lawson}}, but there are topological groups which have no proper  semitopological semigroup compactifications~{\rm\cite{Megr2001}}. 

(b) If in the Definition~\ref{DEF} the requirement on $s$ is relaxed to be a continuous homomorphism, then one gets the notion of an {\it enveloping semigroup of $G$} in~{\rm\cite[\S\ 3.6]{Vries}}.

(b') If in the Definition~\ref{DEF}  $G$ is a semigroup and a topological space, the requirement on $s$ is relaxed to be a continuous homomorphism  and the condition of continuity of the  restriction of the multiplication to $s(G)\times S$ is relaxed to its continuity on the left, then one gets the notion of a {\it semigroup compactification} in~{\rm\cite{HS}}.

(c)  $\mathfrak{E} (G)$ with the order: $(S_1, s_1)\leq (S_2, s_2)$ if there is a continuous homomorphism $\psi: S_2\to S_1$ such that $\psi\circ s_2=s_1$,  is a poset. Below it will be shown that the poset $\mathfrak{E} (G)$ is nonempty.}
\end{rem}

Let $(S, s)\in \mathfrak{E} (G)$. The restriction of the multiplication to $s(G)\times S$ defines a continuous action $\alpha_S: G\times S\to S$ ($\alpha_S(g, x)=s(g)\bullet x$) and the diagram
$$\begin{array}{ccccc}
\quad G\times G & \stackrel{\id\times s}\hookrightarrow & G\times S &  \stackrel{s\times\id}\hookrightarrow    & S\times S\\
\cdot  \downarrow &   & \quad \downarrow \alpha_S   &  &  \quad \downarrow \bullet \\
\quad G & \stackrel{s}\hookrightarrow & S &  \stackrel{\id} {\longrightarrow} & S
\end{array}$$
is commutative. 

The unique uniformity $\tilde{\mathcal E}$ on $S$ induces a totally bounded equiuniformity $\mathcal E$ on $G$ (the completion of $(G, \mathcal E)$  is $(S, \tilde{\mathcal E})$). Denote by $\mathbb{BU}(G)$ the set of totally bounded equiuniformities on $G$ as a phase space of $(G, G, \cdot)$.

\begin{df}\label{def1}
A totally bounded equiuniformity $\mathcal E$ on $G$ is {\it an Ellis equiuniformity} if $\tilde G$ is a proper enveloping Ellis semigroup compactification of $G$ {\rm(}$(\tilde G, \tilde{\mathcal E})$ is the completion of $(G, \mathcal E)${\rm)}.

$\mathbb{E} (G)$ is a poset of Ellis equiuniformities on  $G$.
\end{df}

\begin{rem} {\rm Not every $G$-compactification $(b G, b)$ of  $G$ is a proper enveloping Ellis semigroup compactification of $G$ (equivalently, not every equiuniformity on $G$ is an Ellis equiuniformity).}
\end{rem}

\begin{pro}\label{isomorph}
The posets  $\mathfrak{E} (G)$ and $\mathbb{E} (G)$ are isomorphic.
\end{pro}

\begin{proof} The restriction of isomorphism of posets of compactifications of $X$ and totally bounded uniformities on $X$~\cite[Ch.8,\ \S\ 8.4]{Engelking}  to the proper enveloping Ellis semigroup compactifications of $G$ and the Ellis equiuniformities on $G$ will work if one shows that a continuous map $\psi: S_2\to S_1$ of the  proper enveloping Ellis semigroup compactifications $(S_1, s_1)\leq (S_2, s_2)$ of $G$ such that $\psi\circ s_2=s_1$ is a homomorphism. 

We must check fulfillment of the equality
$$\psi(x\bullet y)=\psi(x)\bullet\psi(y),\ \forall\ (x, y)\in S_2\times S_2.\eqno{(e1)}$$

From commutativity of the diagram 
$$\begin{array}{ccl}
S_2 & \stackrel{\psi}{\longrightarrow} & \quad S_1 \\
\quad s_2 \nwarrow &   &  \nearrow s_1 \\
 & G & 
\end{array}$$
where $s_2, s_1$ are topological isomorphisms,  for $x, y\in G$ one has
$$\psi(s_2(x)\bullet s_2(y))=\psi(s_2(x\cdot y))=s_1(x\cdot y)=s_1(x)\bullet s_1(y)=\psi(s_2(x))\bullet\psi(s_2(y))$$
and (e1) holds for $(x, y)\in s_2(G)\times s_2(G)$.

The restrictions of the maps $$\Psi: S_2\times S_2\to S_1,\ \Psi (x, y)=\psi(x\bullet y),\ \mbox{and}$$  
$$\Psi': S_2\times S_2\to S_1,\ \Psi' (x, y)=\psi(x)\bullet\psi(y),$$
on the subset $s_2(G)\times S_2$ are continuous. Hence, (e1) holds for $(x, y)\in s_2(G)\times S_2$ (since continuous maps coincide on a dense subset $s_2(G)\times s_2(G)$).

Since $S_2$ is a right topological semigroup, for any fixed $y\in S_2$ the restrictions of the maps $\Psi$ and $\Psi'$ to $S_2\times\{y\}$ are continuous and coincide on a dense subset  $s_2(G)\times\{y\}$. Hence, (e1) holds on $S_2\times\{y\}$, $y\in S_2$. Thus (e1) holds on $S_2\times S_2$.
\end{proof}


\subsection{Diagonal action, agreement of algebraic structures with topology}\label{digonal}

(A) For a family of sets $X_t$ with actions $\alpha_t$ of a group $G$, $t\in\mathcal T$, the {\it diagonal action} $\alpha_{\Delta\mathcal T}$ of $G$ on the product $\prod\limits_{t\in\mathcal T}  X_t$
$$\alpha_{\Delta\mathcal T} (g, f)(t)=\alpha_t (g, f(t)),\ t\in\mathcal T,\eqno{(\Delta)}$$
where $f=(f(t))_{t\in\mathcal T}\in\prod\limits_{t\in\mathcal T}  X_t$ (note, that $f(t)=\pr_t (f)$)
is correctly defined (see, for example, \cite[\S\ 2.7 Constructions,\ 2]{Vries}). 

\begin{pro}\label{diagactalg}
{\rm (1)} For an action $\alpha: G\times X\to X$ the map 
$$\imath: G\to X^X,\ \imath (g)(x)=\alpha (g, x),\ x\in X,\eqno{(\imath)}$$
is an isomorphism of $G$ onto the group $\imath (G)$ which is a subsemigroup of $X^X$.

{\rm (2)} $\alpha_{\Delta X} (g,  f)=\imath (g)\circ f$.

{\rm (3)} The following diagram is commutative {\rm(}$\imath$ is a $G$-map {\rm(}not continuous{\rm))}
$$\begin{array}{ccc}
\quad G\times G & \stackrel{{\id\times \imath}}\hookrightarrow & G\times X^X \\
 \cdot  \downarrow &   & \quad \downarrow \alpha_{\Delta X}\\
\quad G & \stackrel{\imath}\hookrightarrow & X^X.
\end{array}$$
\end{pro}

\begin{proof} (1) Since the action is effective $\imath$ is injective. $\forall\ x\in X$ 
$$(\imath (g)\circ \imath (h))(x)=\imath (g)(\imath (h)(x))\stackrel{(\imath)}{=}\imath (g)(\alpha (h, x)))\stackrel{(\imath)}{=}\alpha (g, \alpha(h, x)))=\alpha (g\cdot h, x)\stackrel{(\imath)}{=}\imath (g\cdot h)(x).$$
Thus, $\imath (g\cdot h)=\imath (g)\circ \imath (h)$ and on $\imath (G)$ the stucture of a group (with multiplication $\circ$) is correctly defined which is a subsemigroup of $X^X$  (see \S~\ref{semgr} and~\ref{funcsp}). 

(2) Let $f=(f(x))_{x\in X}$, $g\in G$.  $\forall\ x\in X$ 
$$\alpha_{\Delta X} (g,  f)(x)\stackrel{(\Delta)}{=}\alpha (g, f(x))\stackrel{(\imath)}{=}\imath (g) (f(x))=(\imath (g)\circ f)(x).$$
Thus, $\alpha_{\Delta X} (g,  f)=\imath (g)\circ f$.

(3) $\forall\ g, h\in G$ 
$$\imath (g\cdot h)\stackrel{(1)}{=}\imath (g)\circ\imath (h)\stackrel{(2)}{=}\alpha_{\Delta X}(g, \imath (h))=\alpha_{\Delta X}((\id\times \imath) (g, h)).$$
\end{proof}

It is easy to check that if $(G, X_t, \curvearrowright)$ is a $G$-space ($G$-Tychonoff and $\mathcal U_t$ is the correspondent equiuniformity on $X_t$), $t\in\mathcal T$, then $(G, \prod\limits_{t\in\mathcal T} X_t, \alpha_{\Delta\mathcal T})$ is a $G$-space ($G$-Tychonoff and the Cartesian product $\mathcal U^p$ of equiuniformities  $\mathcal U_t$,  $t\in\mathcal T$,  is an equiuniformity on $\prod\limits_{t\in\mathcal T}  X_t$). 

\begin{thm}\label{diagactalgcont} {\rm I.} Let  $(G=(G, \tau_p), X, \alpha)$ be a $G$-space. Then

{\rm (1)}  $(X^X, \circ)$ is a right topological semigroup  {\rm(}besides, multiplication on the left is continuous for $f\in C(X)${\rm)}, 

{\rm (2)} $\imath: G\to X^X$ is a topological isomorphism of $G$ onto the topological group $\imath (G)$  which is a subsemigroup of $X^X$  and the restriction of multiplication $\circ$ to $\imath (G)\times X^X$ is continuous {\rm(}besides, the multiplication on the left is  continuous for $f\in\imath (G)$, $f\in C(X)${\rm)}, 

{\rm (3)} $\imath: G\to X^X$  is a continuous $G$-map, 

{\rm (4)} the closure $\cl~\imath (G)$ is an invariant subset of $X^X$ {\rm(}and, hence,  $(G, \cl~\imath (G), \alpha_{\Delta X}|_{G\times\cl~\imath (G)})$ is a $G$-space{\rm)} and a subsemigroup of $X^X$. The restriction of multiplication $\circ$ to $\imath (G)\times\cl~\imath (G)$ is continuous {\rm(}besides, the  multiplication on the left is  continuous for $f\in\cl~\imath (G)$, $f\in C(X)${\rm)}. 

\medskip

{\rm II.} Let  $(G=(G, \tau_p), X, \alpha)$ be a  $G$-Tychonoff space, $\mathcal U_{X}$ is an equiuniformity on $X$ and $(\tilde X, \tilde{\mathcal U}_X)$ is the completion of $(X, \mathcal U_X)$. Then

{\rm (a)} $\mathcal U_X^p|_{\imath (G)}$ is an equiuniformity on $G$, the completion of $(G, \mathcal U_X^p|_{\imath (G)})$ is the closure of $\imath (G)$ in $\tilde X^X$, 

{\rm (b)} for the  $G$-Tychonoff space  $(G=(G, \tau_p), \tilde X, \tilde\alpha)$  $\mathcal U_{\tilde X}^p|_{\tilde\imath (G)}$ is an equiuniformity on $G$,  the completion of $(G, \mathcal U_{\tilde X}^p|_{\tilde\imath (G)})$ is the closure of $\tilde\imath (G)$ in $\tilde X^{\tilde X}$, 

{\rm (c)}  $\mathcal U_X^p|_{\imath (G)}\subset\mathcal U_{\tilde X}^p|_{\tilde\imath (G)}$. 

\medskip

{\rm III.}  Let  $(G=(G, \tau_p), X, \alpha)$ be a  $G$-Tychonoff space, $\mathcal U_{X}\in\mathbb{BU}(X)$ and $(b X, \tilde{\mathcal U}_X)$ is the completion of $(X, \mathcal U_X)$. Then

{\rm (a')} $\mathcal U_X^p|_{\imath (G)}\in\mathbb{BU} (G)$, the completion of $(G, \mathcal U_X^p|_{\imath (G)})$ {\rm(}the compactification of $G${\rm)} is the closure of $\imath (G)$ in $(b X)^X$, 

{\rm (b')} for the  $G$-Tychonoff space  $(G=(G, \tau_p), b X, \tilde\alpha)$  $\mathcal U_{b X}^p|_{\tilde\imath (G)}\in\mathbb{E} (G)$, the completion of $(G, \mathcal U_{b X}^p|_{\tilde\imath (G)})$ {\rm(}the proper enveloping Ellis semigroup compactification of $G${\rm)} is the closure of $\tilde\imath (G)$ in $(b X)^{b X}$ {\rm(}$\cl~\tilde\imath (G)\in\mathfrak{E}(G)${\rm)}, 

{\rm (c')} $\mathcal U_X^p|_{\imath (G)}\subset\mathcal U_{b X}^p|_{\tilde\imath (G)}$, $(\cl~\imath (G), \imath)\leq (\cl~\tilde\imath (G), \tilde\imath)$. 

\medskip

{\rm IV.} Equiuniformity $\mathcal U_X^p|_{\imath (G)}$  {\rm(}respectively, $\mathcal U_{\tilde X}^p|_{\tilde\imath (G)}${\rm)} is the least uniformity $\mathcal U$ on $G$ such that the orbit maps $\alpha_x: (G, \mathcal U)\to (X, \mathcal U_X)$, $x\in X$ {\rm(}respectively,  $\tilde\alpha_x: (G, \mathcal U)\to (\tilde X, \tilde{\mathcal U}_X)$, $x\in\tilde X${\rm)}  are uniformly continuous. Hence, $\mathcal U_X^p|_{\imath (G)}$ {\rm(}respectively,  $\mathcal U_{\tilde X}^p|_{\tilde\imath (G)}${\rm)} is the initial uniformity with respect to the correspondent family of maps.
\end{thm}

\begin{proof} {\rm I}. (1) is shown in  \S~\ref{funcsp}.

(2) From \S~\ref{admtop} and Proposition~\ref{diagactalg} one has that $\imath$ is a topological isomorphism of $G$ onto the subsemigroup of $X^X$. The diagram 
$$\begin{array}{ccl}
G\times X^X & \stackrel{\imath\times\id}{\longrightarrow} & \imath (G)\times X^X \\
\quad \alpha_{\Delta X} \searrow &   &  \swarrow \circ \\
 & X^X & 
\end{array}$$
is commutative by item (2) of Proposition~\ref{diagactalg} and $\imath\times\id$ is a homeomorphism. Continuity of $\alpha_{\Delta X}$ yields continuity of $\circ$. (Continuity of the diagonal action $\alpha_{\Delta X}: G\times X^X\to X^X$ is equivalent to continuity of the multiplication $\circ: \imath (G)\times X^X\to X^X$.)

(3) follows from item (3) of Proposition~\ref{diagactalg}.

(4) Since the diagonal action $ \alpha_{\Delta X}$ is continuous and $\imath (G)$ is an invariant subset  of $X^X$ the closure $\cl~\imath (G)$ is an invariant subset of $X^X$ and, hence, by Proposition~\ref{diagactalg}, 
$$\imath (G)\circ\cl~\imath (G)=\bigcup\{\imath (g)\circ f\ |\ g\in G, f\in \cl~\imath (G)\}=\bigcup\{\alpha_{\Delta X} (g, f)\ |\ g\in G, f\in\cl~\imath (G)\}=\cl~\imath (G).$$

Since $\imath (G)\circ f\subset\cl~\imath (G)$, $f\in \cl~\imath (G)$, and $X^X$ is a right topological semigroup, $\cl~\imath (G)\circ f\subset\cl~\imath (G)$, $f\in\cl~\imath (G)$. Hence, 
$$\cl~\imath (G)\circ\cl~\imath (G)=\cl~\imath (G)\ \mbox{and}\ \cl~\imath (G)\ \mbox{is a subsemigroup of}\  X^X.$$
The second statement follows from item (2).

\medskip

{\rm II} (a), {\rm III}  (a').  The restriction  $\mathcal U_X^p|_{\imath (G)}$ of the equiuniformity $\mathcal U_X^p$  on the invariant subset $\imath (G)$ is an equiuniformity on $\imath (G)$. If $\mathcal U_X\in\mathbb{BU}_X$, then $\mathcal U_X^p$  and its restriction $\mathcal U_X^p|_{\imath (G)}$ are totally bounded. The equivalence of $G$-spaces $(G, G, \cdot)$ and $(G, \imath (G), \alpha_{\Delta X}|_{G\times\imath (G)})$ and the fact that the completions of uniform subspaces are their closures in the completions of the correspondent uniform spaces are equal finishes the proof. 

The proofs of {\rm II} (b) and {\rm III} (b') are the same as of {\rm II} (a) and {\rm III} (a'), taking into consideration the fact that  $\tilde X^{\tilde X}$ and $(b X)^{b X}$ are complete spaces, semigroups and by {\rm I} (4)  the closures of $G$ are invariant subsets. 

{\rm III} (c) follows from the uniform continuity of the projection of $(\tilde X, \tilde{\mathcal U}_X)^{\tilde X}$ onto $(\tilde X, \tilde{\mathcal U}_X)^{X}$,  that  $(X, \mathcal U_X)^{X}$ is a uniform subspace of  $(\tilde X, \tilde{\mathcal U}_X)^{X}$ and commutativity of the diagram
$$\begin{array}{ccl}
(\tilde X, \tilde{\mathcal U}_X)^{\tilde X} \stackrel\pr {\longrightarrow} &  (\tilde X, \tilde{\mathcal U}_X)^{X} & \hookleftarrow X^X \\
\imath' \nwarrow & \tilde\imath \uparrow  &  \nearrow \imath \\
 & G & 
\end{array}$$
{\rm III} (c') is proved similarly. The inclusion  $\mathcal U_X^p|_{\imath (G)}\subset\mathcal U_{b X}^p|_{\tilde\imath (G)}$  yields the map of the corresponding compactifications. 

\medskip

{\rm IV}. For arbitrary $x\in X$ and ${\rm U}\in\mathcal U_X$ take ${\rm V}=\{(g, h)\in G\times G\ |\ (\alpha (g, x), \alpha (h, x))\in {\rm U}\}\in\mathcal U_X^p|_{\imath (G)}$. If $(g, h)\in {\rm V}$, then $(\alpha (g, x), \alpha (h, x))\in {\rm U}$ and $\alpha_x: (G, \mathcal U_X^p|_{\imath (G)})\to (X, \mathcal U_{X})$ is uniformly continuous.  The same argument works in the case of the completion $\tilde X$. 

If $\mathcal U$ is an uniformity on $G$ and the orbit maps $\alpha_x: (G, \mathcal U)\to (X, \mathcal U_{X})$, $x\in X$, are uniformly continuous, then the map $\imath: (G, \mathcal U)\to (\imath (G), \mathcal U_X^p|_{\imath (G)})$ is uniformly continuous. Thus,  $\mathcal U_X^p|_{\imath (G)}\subset\mathcal U$. The same argument works in the case of the completion $\tilde X$.
\end{proof}

\bigskip

(B) Let $(G=(G, \tau_p), X, \alpha)$ be a $G$-space, $q_{\sigma}: G\to G/\st_{\sigma}$ is the quotient map, $\sigma\in\Sigma_X$.

The following diagram (the maps are continuous) is commutative for any $\sigma\in\Sigma_X$
$$\begin{array}{ccc}
\quad G & \stackrel{\imath}\hookrightarrow & X^X  \\
 q_{\sigma}  \downarrow &   & \quad \downarrow \pr_{\sigma}\\
\quad G/\st_{\sigma} & \stackrel{\varphi_{\sigma} }\longrightarrow & X^{\sigma}, 
\end{array}\eqno{(d)}$$
where $\varphi_{\sigma}  (g \st_{\sigma})(x)=\alpha (g, x)$, $x\in\sigma$ ($\varphi_{\sigma}$ is correctly defined since if $g'=gh$, $h\in\st_{\sigma}$, then  $\alpha (g', x)=\alpha (gh, x)=\alpha (g, \alpha (h, x))=\alpha (g, x)$). Considering all spaces as $G$-spaces (with the action of $G$ on $G$ and $G/\st_{\sigma}$ by multiplication on the left and the diagonal actions (from the action $\alpha$) on products) the maps in the diagram $(d)$ are $G$-maps. 

If  $(G=(G, \tau_p), X, \alpha)$ is a $G$-Tychonoff space and $\mathcal U_X$ is an equiuniformity on $X$, then the uniformity  $\mathcal U_X^p|_{\imath (G)}$ on $\imath (G)$ has a spectral representation. It is initial with respect to the restrictions to $\imath (G)$ of the projections $\pr_{\sigma}: X^X\to (X^{\sigma}, \mathcal U_X^{\sigma})$ (products $X^{\sigma}$ with  the Cartesian product of equiuniformities $\mathcal U_X$). 

$(X^X, \mathcal U_X^p)$ is uniformly equivalent to the {\it inverse limit}\ $\lim\limits_{\longleftarrow} S$ of the inverse spectrum $S=\{(X^{\sigma}, \mathcal U_X^{\sigma}), \pr^{\sigma'}_{\sigma}, \Sigma_X\}$ (the proof is standard and the case of homeomorphism is in~\cite[Example~2.5.3]{Engelking}). Therefore, the completion of $(G, \mathcal U_X^p|_{\imath (G)})$ is uniformly equivalent to the limit of the inverse spectrum $\{\widetilde{\pr_{\sigma} (\imath (G))}, \widetilde{\pr^{\sigma'}_{\sigma}}, \Sigma_X\}$, where $(\widetilde{\pr_{\sigma} (\imath (G))}, \widetilde{\mathcal U}_X^{\sigma}|_{\pr_{\sigma} (\imath (G))})$ is the completion of $(\pr_{\sigma} (\imath (G)), \mathcal U_X^{\sigma}|_{\pr_{\sigma} (\imath (G))})$, 
$\widetilde{\pr^{\sigma'}_{\sigma}}:\widetilde{\pr_{\sigma'} (\imath (G))}\to\widetilde{\pr_{\sigma} (\imath (G))}$ is the extension of the restriction to $\pr_{\sigma'}(\imath(G))$  of the projection $\pr^{\sigma'}_{\sigma}: X^{\sigma'}\to X^{\sigma}$, $\sigma\subset\sigma'$.

The maximal equiuniformity $\mathcal U_{G/\st_{\sigma}}$ on $G/\st_{\sigma}$ is the quotient of the right uniformity $R$ on $G$ (follows from~\cite{RD} and~\cite{ChK}).

\begin{pro}\label{lemincl}
Let $(G=(G, \tau_p), X, \alpha)$ be a $G$-Tychonoff space, $\mathcal U_X$ is an equiuniformity on $X$. Then for any $\sigma\in\Sigma_X$ the map $\varphi_{\sigma}: (G/\st_{\sigma}, \mathcal U_{G/\st_{\sigma}})\to ((\pr_{\sigma}\circ\imath) (G), \mathcal U_X^{\sigma}|_{\pr_{\sigma}(\imath (G))})$ is a uniformly continuous bijection. Hence, 
$$\mathcal U_X^p|_{\imath (G)}\subset R_{\Sigma_X}\ (\mbox{identifying}\ G=\imath (G)).$$ 
\end{pro}

\begin{proof} Immediately follows from the commutativity of the diagram $(d)$, the condition that the restriction $\mathcal U^{\sigma}|_{\pr_{\sigma}(\imath (G))}$ is an equiuniformity which is the restriction of the quotient uniformity of $\mathcal U_X^p$ and that the maximal equiuniformity on $G/\st_{\sigma}$ is the quotient of the right uniformity $R$ on $G$.

Since $ R_{\Sigma_X}$ is initial with respect to the quotient maps $q_{\sigma}: G\to (G/\st_{\sigma}, \mathcal U_{G/\st_{\sigma}})$,  see \S~\ref{uniftopgr}, and $\mathcal U_X^p|_{\imath (G)}$  is initial with respect to the restrictions to $\imath (G)$ of the projections $\pr_{\sigma}: X^X\to (X^{\sigma}, \mathcal U_X^{\sigma})$ one has $\mathcal U_X^p|_{\imath (G)}\subset R_{\Sigma_X}$. 
\end{proof}

\begin{cor}\label{coinunif1} Let $(G=(G, \tau_p), X, \alpha)$ be a $G$-Tychonoff space, $\mathcal U_X$ is an equiuniformity on $X$. If the uniformities on $G$ which are initial with respect to the quotient maps $q_{\sigma}: G\to (G/\st_{\sigma}, \mathcal U_{G/\st_{\sigma}})$, $\sigma\in\Sigma_X$, and the maps $\pr_{\sigma}|_{\imath (G)}\circ\imath: G\to (X^{\sigma}, \mathcal U_X^{\sigma})$ correspondently, coincide {\rm(}in particular, 
$$\begin{array}{c}
\mbox{for any}\ \sigma\in\Sigma_X  \\
\varphi_{\sigma}: (G/\st_{\sigma}, \mathcal U_{G/\st_{\sigma}})\to ((\pr_{\sigma}\circ\imath) (G), \mathcal U_X^{\sigma}|_{(\pr_{\sigma}\circ\imath) (G)})\ \mbox{is a uniform equivalence} {\rm)}, 
\end{array}\eqno{(op)}$$
then

{\rm (a)} $\mathcal U_X^p|_{\imath (G)}=R_{\Sigma_X}\ (\mbox{identifying}\ G=\imath (G))$,

{\rm (b)} for any $\mathcal U_X\subset\mathcal U'_X$ {\rm(}in particular, the maximal equiuniformity on $X${\rm)} ${\mathcal U'}_X^p|_{\imath (G)}=R_{\Sigma_X}\ (\mbox{identifying}\ G=\imath (G))$ {\rm(}in particular, the map  $\varphi_{\sigma}: (G/\st_{\sigma}, \mathcal U_{G/\st_{\sigma}})\to (\pr_{\sigma}(\imath (G)), {\mathcal U'}_X^{\sigma}|_{\pr_{\sigma}(\imath (G))})$, $\sigma\in\Sigma_X$,  is a uniform equivalence{\rm)}. 
\end{cor}

\begin{rem}{\rm Since the quotient maps  $q_{\sigma}: G\to G/\st_{\sigma}$ are open, if the condition (op) of Corollary~\ref{coinunif1}  is valid, then the restrictions of projections  $\pr_{\sigma}|_{\imath (G)}$, $\sigma\in\Sigma_X$,  are open.}
\end{rem}

\bigskip

(C) Let $(G, X, \curvearrowright)$ be a $G$-Tychonoff space,  $\mathcal U_X$ is an equiuniformity on $X$, $\mathcal U_X^u$ is a u.u.c. and $\tau_u$ is the induced  t.u.c. on $X^X$.  Evidently, ${\mathcal U}_X^{p}\subset\mathcal U_X^u$ and $\tau_u\geq\tau_p$.

It is easy to check that the diagonal action $\alpha_{\Delta X}: G\times (X^X, \tau_u)\to (X^X, \tau_u)$ is continuous, $\mathcal U_X^{u}$ is  an equiuniformity on $X^X$. Its restriction $\mathcal U_X^{u}|_{\imath(G)}$ on the invariant subset $G=\imath (G)$ is an equiuniformity on $G$ and the induced topology $\tau_u$ is the t.u.c. 

\begin{lem}\label{l2}
Let $((G, \tau), X, \curvearrowright)$ be a $G$-Tychonoff space. For every equiuniformity $\mathcal U_X$ on $X$
$$\tau\geq\tau_u.$$
If $\tau=\tau_p$, then $\tau=\tau_u$. 
\end{lem}

\begin{proof} The restriction of the diagonal action  on the invariant subset $\imath (G)$ is  continuous. Identifying $G=\imath (G)$ by Proposition~\ref{diagactalg} one has that the action $\cdot$ of $(G, \tau)$ on  $(G, \tau_u)$ is continuous. Hence, $\tau\geq\tau_u$. 

Since $\tau_u\geq\tau_p$, if $\tau=\tau_p$, then $\tau_u=\tau$.
\end{proof}

\begin{rem}
{\rm The u.u.c. $\mathcal U_X^{u}|_{\imath(G)}$ is a right uniformity $R$  on $G$~{\rm\cite{RD}}. The base of $N_G(e)$ is formed by internals of the sets 
$$O_{{\rm U}}=\{g\in G\ |\ (x, \alpha (g, x))\in  {\rm U},\ x\in X\},\ {\rm U}\in\mathcal U_X.$$
$O^{-1}=O$, since $(\alpha (g^{-1}, x), x)\in {\rm U}$, $x\in X$, iff  $(y, \alpha (g, y))\in {\rm U}$, $y{\rm(}=\alpha (g^{-1}, x){\rm)}\in X$.}
\end{rem}

\begin{cor} 
Let $((G, \tau_p), X, \curvearrowright)$ be a $G$-Tychonoff space, $\mathcal U_X$ is an equiuniformity on $X$. Then for any $g\in G$ and any ${\rm U}\in\mathcal U_X$ there exist $\sigma\in\Sigma_X$ and ${\rm V}\in\mathcal U_X$ such that 
$$\mbox{if}\ (\alpha (g, x), \alpha (h, x))\in{\rm V},\ x\in\sigma,\ \mbox{then}\  (\alpha (g, x), \alpha (h, x))\in{\rm U},\ x\in X.$$
\end{cor}


\subsection{Ellis equiuniformity on a group from an equiuniformity on a phase space}\label{Elliscmp}

Let $(G=(G, \tau_p), X, \alpha)$ be a $G$-Tychonoff space, $\mathbb{BU}(X)$ is a poset of  totally bounded equiuniformities on $X$. 

For any $\mathcal U_X\in \mathbb{BU}(X)$ by Theorem~\ref{diagactalgcont} (tem III  (a')) $\mathcal U_X^p|_{\imath (G)}\in\mathbb{BU} (G)$ and (item III   (b')) $\mathcal U_{b X}^p|_{\tilde\imath (G)}\in\mathbb{E} (G)$ are defined. Thus, the maps 
$$\mathfrak{E}_{\alpha}(=\mathfrak{E}_{\alpha}(G, X, \alpha)):\mathbb{BU}(X)\to \mathbb{BU} (G),\ \mathfrak{E}_{\alpha}(\mathcal U_X)=\mathcal U_X^p|_{\imath (G)}\ \mbox{and}$$ 
$$\tilde{\mathfrak{E}}_{\alpha}(=\tilde{\mathfrak{E}}_{\alpha}(G, X, \alpha)):\mathbb{BU}(X)\to \mathbb{E} (G),\ \tilde{\mathfrak{E}}_{\alpha}(\mathcal U_X)=\mathcal U_{b X}^p|_{\tilde\imath (G)}$$ are defined. The completion of $(G, \tilde{\mathfrak{E}}_{\alpha}(\mathcal U_X))$ is $(E(G, \mathcal U_X), \widetilde{\tilde{\mathfrak{E}}_{\alpha}(\mathcal U_X)})$, where $E(G, \mathcal U_X)=\cl~\tilde\imath (G)$  is the proper enveloping Ellis semigroup compactification.

\begin{pro}\label{order} 
Let $(G=(G, \tau_p), X, \alpha)$ be a $G$-Tychonoff space. For any $\mathcal U_X\in\mathbb{BU}(X)$ $\mathfrak{E}_{\alpha}(\mathcal U_X)\subset\tilde{\mathfrak{E}}_{\alpha} (\mathcal U_X)$. 

Let $(bX, \tilde{\mathcal U}_X)$ be the completion of $(X, \mathcal U_X)$. $\mathfrak{E}_{\alpha}(\mathcal U_X)=\tilde{\mathfrak{E}}_{\alpha} (\mathcal U_X)$ iff 
$$\begin{array}{l}
\mbox{for any}\ y\in bX\setminus X\ \mbox{and any}\ \tilde{\rm U}\in\tilde{\mathcal U}_X\ \mbox{there exists}\ \sigma\in\Sigma_X\ \mbox{and}\ \tilde{\rm V}\in\tilde{\mathcal U}_X \\
\mbox{such that if}\ (\tilde\alpha (g, x), \tilde\alpha (h, x))\in \tilde{\rm V}, x\in\sigma,\  \mbox{then}\  (\tilde\alpha (g, y), \tilde\alpha (h, y))\in\tilde{\rm U}.
\end{array}\eqno{(\star)}$$

The maps $\mathfrak{E}_{\alpha}:\mathbb{BU}(X)\to \mathbb{BU} (G)$ and $\tilde{\mathfrak{E}}_{\alpha}:\mathbb{BU}(X)\to \mathbb{E} (G)$
preserve partial orders.
\end{pro}

\begin{proof} $\mathfrak{E}_{\alpha}(\mathcal U_X)\subset\tilde{\mathfrak{E}}_{\alpha} (\mathcal U_X)$ by item III (c') of Theorem~\ref{diagactalgcont}.

Condition $(\star)$ is a reformulation of the fact that the restriction of projection $\pr: (b X)^{b X}\to (b X)^X$ onto $\tilde\imath (G)$ is a uniform equivalence. 

Let $\mathcal U_X\subset\mathcal U'_X\in \mathbb{BU}(X)$. For the $G$-compactifications $(b X=b_{\mathcal U_X} X, b=b_{\mathcal U_X})$ and $(b' X=b_{\mathcal U'_X} X, b'=b_{\mathcal U'_X})$ of $X$ (completions  of $(X, \mathcal U_X)$ and $(X, \mathcal U'_X)$ respectively) let a $G$-map $\varphi: b' X\to b X$ be such that $\varphi\circ b'=b$ (the map of $G$-compactifications). The restriction $\varphi|_{b'(X)(=X)}: (X, \mathcal U'_X)\to  (X, \mathcal U_X)$ is uniformly continuous.

The Cartesian product $\varphi^{X}:  b' X^X\to  b X^X$ of maps $\varphi$ of compacta is uniformly continuous. Hence, the restriction $\varphi^X|_{\imath'(G)(=G)}: (G, \mathfrak{E}_{\alpha}(U'_X))\to  (G, \mathfrak{E}_{\alpha}(U_X))$  is uniformly continuous and $\mathfrak{E}_{\alpha}(U_X)\subset\mathfrak{E}_{\alpha}(U'_X)$.

Let $\tilde\alpha: G\times b X\to b X$  and $\tilde\alpha': G\times b' X\to b' X$  be the extensions of $\alpha: G\times X\to X$, $X'\subset  b' X$ is any subset the restriction of $\varphi$ to which is a bijection. The Cartesian product $\varphi^{b' X}:   (b' X)^{b' X}\to  (b X)^{b' X}$ of maps $\varphi$ is a $G$-map. The projection $\pr_{X'}: (b X)^{b' X}\to  (b X)^{X'}$ is  a $G$-map, the map $\jmath: (b X)^{X'}\to (b X)^{b X}$,  $\jmath (f)=h$, where $h(t)=f({\varphi^{-1}(t)})$, $t\in b X$ (identification of coordinates) is also a $G$-map. Their composition $\psi=\jmath\circ\pr_{X'}\circ\varphi^{b' X}:  (b' X)^{b' X}\to  (b X)^{b X}$ is a $G$-map.

The following diagram 
$$\begin{array}{ccl}
(b' X)^{b' X} & \stackrel{\psi}{\longrightarrow} &  (b X)^{b X} \\
\imath' \nwarrow &   &  \nearrow \imath \\
 & G & 
\end{array}$$
is commutative.  In fact, for any  $x\in X$ 
$$\imath (g)(b(x))=\tilde\alpha (g, b(x))=\alpha (g, b(x))\ \mbox{and}$$ 
$$[(\psi\circ\imath') (g)] (b(x))=\pr_{b(x)}\big((\psi\circ\imath') (g)\big)=\pr_{b(x)}\big((\jmath\circ\pr_{X'}\circ \varphi^{b' X}\circ \imath')(g)\big)=$$
since $\pr_{b(x)}\circ\jmath\circ\pr_{X'}=\pr_{b'(x)}$
$$=\pr_{b'(x)}\big((\varphi^{b' X}\circ \imath')(g)\big)=\varphi\big(\imath' (g)(b'(x))\big)=\varphi\big(\tilde\alpha'(g, b'(x))\big)=$$
since $\varphi$ is a $G$-map 
$$=\tilde\alpha(g, \varphi (b'(x)))=\tilde\alpha(g, b(x))\stackrel{x\in X}=\alpha(g, b(x)).$$
Thus, $\imath (g)(t)=(\psi\circ\imath') (g)(t)$, $t\in b (X)$.  Being continuous they coincide on $b X$ and  $\imath (g)=(\psi\circ\imath') (g)$.

Continuity of $\psi$,  commutativity of the above diagram and compactness (by Theorem~\ref{diagactalgcont}) of  $\cl~\imath' (G)=E (G, \mathcal U'_X)$, $\cl~\imath (G)=E (G, \mathcal U_X)$ yields that the restriction $\psi|_{E (G, \mathcal U'_X)}: E (G, \mathcal U'_X)\to E (G, \mathcal U_X)$ is the map of $G$-compactifications. From the proof of Proposition~\ref{isomorph} one has $E (G, \mathcal U'_X)\geq E (G, \mathcal U_X)$. It remains to apply Proposition~\ref{isomorph} to obtain $\tilde{\mathfrak{E}}_{\alpha}(\mathcal U_X)\subset\tilde{\mathfrak{E}}_{\alpha}(\mathcal U'_X)$.
\end{proof}

\begin{rem}  {\rm The map $\psi$ of proper Ellis semigroup compactifications in Proposition~\ref{order}  doesn't depend on the choice of $X'$ in Proposition~{\rm\ref{order}}.

In general, the maps $\mathfrak{E}_{\alpha}:\mathbb{BU}(X)\to \mathbb{BU} (G)$ and $\tilde{\mathfrak{E}}_{\alpha}:\mathbb{BU}(X)\to \mathbb{E} (G)$  are neither injective, no surjective.

If $X$ is compact, then $\mathfrak{E}_{\alpha}(\mathcal U_X)=\tilde{\mathfrak{E}}_{\alpha} (\mathcal U_X)$ for the unique uniformity $\mathcal U_X$ on $X$.}
\end{rem}

\begin{pro}~{\rm(\cite[Theorem 2.1]{Sorin1}, \cite[Theorem 1]{Sorin})}\label{a2-2}
Let $(G=(G, \tau_p), X, \alpha)$ be a $G$-Tychonoff space and $\mathcal U_{X}\in\mathbb{BU}(X)$.   

If for any  $\sigma\in\Sigma_{X}$ and for any entourage ${\rm U}\in {\mathcal U}_X$ there is an entourage ${\rm V}={\rm V}(\sigma; {\rm U})\in {\mathcal U}_X$ such that the condition is met:
$$\begin{array}{c}
\mbox{\rm if for}\ g, h\in G\ (\alpha(g, x), \alpha(h, x))\in {\rm V}\ \mbox{\rm holds for}\ x\in\sigma, \\
\mbox{\rm then there exists}\ \ g'\in g\st_{\sigma}\ \ \mbox{\rm such that}\ \ h\in O_{{\rm U}}g',\\ 
\end{array} \eqno{(\star\star)}$$
where $O_{{\rm U}}=\{f\in G\ |\ (\alpha(f, x), x)\in{\rm U},\ x\in X\}$.
Then $\mathfrak{E}_{\alpha} (\mathcal U_X)=R_{\Sigma_{X}}$.
\end{pro}
 
\begin{proof}
By Corollary~\ref{coinunif1} it is sufficient to show that for any $\sigma\in\Sigma_{X}$  the map  $\varphi_{\sigma}: (G/\st_{\sigma}, \mathcal U_{G/\st_{\sigma}})\to ((\pr_{\sigma}\circ\imath) (G), {\mathcal U}_X^{\sigma}|_{(\pr_{\sigma}\circ\imath) (G)})$ is a uniform equivalence.
Since  $\mathfrak{E}_{\alpha} (\mathcal U_X)\subset R_{\Sigma_{X}}$ by Lemma~\ref{lemincl}, only the uniform continuity of $\varphi_{\sigma}^{-1}$ must be checked, i.e. 
in the cover $$\{O_{{\rm U}}g\st_{\sigma}\ |\ g\in G\},\ {\rm U}\in {\mathcal U}_X,$$
the cover 
$$\{U_{\sigma; g; {\rm V}}=\{h\in G\ |\ (\alpha(g, x), \alpha(h, x))\in {\rm V},\ x\in\sigma\}\ |\ g\in G\},$$
is refined for some ${\rm V}\in{\mathcal U}_X$. 
 
By the condition $(\star\star)$ for $\sigma\in\Sigma_{X}$ and ${\rm U}\in {\mathcal U}_X$ there exists ${\rm V}\in {\mathcal U}_X$ such that if for $g, h\in G$ $(\alpha(g, x), \alpha(h, x))\in {\rm V}$, $x\in\sigma$, then there exists $g'\in g\st_{\sigma}$ such that $h\in O_{{\rm U}}g'$. Then for any fixed $g\in G$ and any $h\in U_{\sigma; g; {\rm V}}$ 
$$h\in O_{{\rm U}} g\st_{\sigma}.$$
Therefore, $U_{\sigma; g; {\rm V}}\subset   O_{{\rm U}}g\st_{\sigma}$, i.e. the cover   
$\{U_{\sigma; g; {\rm V}}\ |\ g\in G\}\succ\{O_{{\rm U}}g\st_{\sigma}\ |\ g\in G\}$ and, hence, taking in account item (C) of \S~\ref{digonal}, $\varphi_{\sigma}$ is a uniform equivalence.
\end{proof}

\begin{rem}{\rm The condition $h\in O_{{\rm U}}g$ in Proposition~\ref{a2-2} is equivalent to the condition $(\alpha(h, x), \alpha(g, x))\in {\rm U}$, $x\in X$. 

Indeed, $(\alpha(h, x), \alpha(g, x))\in {\rm U}, x\in X\Longleftrightarrow (\alpha(hg^{-1}, x), x)\in {\rm U}, x\in X \Longleftrightarrow hg^{-1}\in O_{{\rm U}} \Longleftrightarrow  h\in O_{{\rm U}}g$.}
\end{rem}

\begin{cor}\label{cora2-2}
Let $(G=(G, \tau_p), X, \alpha)$ be a $G$-Tychonoff space, $\mathcal U_{X}\in\mathbb{BU}(X)$ and  $(b X, \tilde{\mathcal U}_X)$ is the completion of  $(X, \mathcal U_X)$ {\rm(}$G$-compactification of $X${\rm)}. Then  
\begin{itemize}
\item[{\rm (1)}] $$\begin{array}{ccl}
\mathfrak{E}_{\alpha} (\mathcal U_X) & \subset &  R_{\Sigma_{X}} \\
\cap &   &  \cap \\
 \tilde{\mathfrak{E}}_{\alpha} (\mathcal U_X) & \subset &  R_{\Sigma_{b X}}.
\end{array}$$
\item[{\rm (2)}] $\mathfrak{E}_{\alpha} (\mathcal U_X)=R_{\Sigma_{X}}$ {\rm(}respectively, $\tilde{\mathfrak{E}}_{\alpha} (\mathcal U_X)=R_{\Sigma_{b X}}${\rm)}  if the correspondent condition  $(\star\star)$ is valid. 
\item[{\rm (3)}]  If the condition  $(\star\star)$ is valid for $(G, (b X, \tilde{\mathcal U}_X), \tilde\alpha)$, then  the condition  $(\star\star)$ is valid for $(G, (X, \mathcal U_X), \alpha)$.
\end{itemize}
\end{cor}

\begin{rem}
{\rm From Corollary~\ref{coinunif1} it also follows that if the condition  $(\star\star)$ holds for $\mathcal U_{X}\in\mathbb{BU}_X$, then it holds for the maximal equiuniformity on $X$.}
\end{rem}

\begin{rem} {\rm (1) The above approach of obtaining compactifications of transformation groups is used in~\cite{Sorin} for the group of homeomorphisms of the compactum $K$ in the topology of pointwise convergence. 

(2) Fulfillment of the condition $(\star\star)$ in Proposition~\ref{a2-2} yields Roelcke precompactness of an acting group (in the topology of pointwise convergence) by Remark~\ref{a2-1}. 

Other results that guarantee the Roelcke precompactness of an acting group are the following. 

(a) If a (closed) subgroup $H$ of a topological group $G$ is Roelcke precompact and the maximal equiuniformity on the coset space $G/H$ (with the action of $G$ by multiplication on the left) is totally bounded, then $G$ is Roelcke precompact~\cite[Proposition 6.4]{Megr} ($G$ can be considered as a transformation group of $G/H$).

(b) If a topological group $G$ has a directed family  $\mathcal K$ of small subgroups such that the maximal equiuniformity on a coset space $G/H$ (with the action of $G$ by multiplication on the left) is totally bounded, $H\in\mathcal K$, then  $G$ is Roelcke precompact~\cite[Corollary 4.5]{Kozlov}.

(c) If the action  $G\curvearrowright X$, where $X$ is a discrete space, is oligomorphic, then $(G, \tau_{\partial})$ is Roelcke precompact~\cite {Tsan} (equivalently, for the action of $G$ in the permutation topology (which coincides with the topology of pointwise convergence) the maximal equiuniformity on $X$ is totally bounded~\cite[Theorem 3.3]{Sorin1}). }
\end{rem}

\begin{pro}\label{connection}
Let $(G=(G, \tau_p), X, \alpha)$ and $(G=(G, \tau_p), Y, \alpha')$ be $G$-Tychonoff spaces, $\mathcal U_{X}\in\mathbb{BU}(X)$ and  $(b X, \tilde{\mathcal U}_X)$ is the completion of  $(X, \mathcal U_X)$ {\rm(}$G$-compactification of $X${\rm)}, $\mathcal U_{Y}\in\mathbb{BU}(Y)$ and  $(b Y, \tilde{\mathcal U}_Y)$ is the completion of  $(Y, \mathcal U_Y)$ {\rm(}$G$-compactification of $Y${\rm)}, $\varphi: (X, \mathcal U_{X})\to  (Y, \mathcal U_{Y})$ is a uniformly continuous $G$-map, $\varphi (X)$ is dense in $Y$. Then  
$$\tilde{\mathfrak{E}}_{\alpha'} (\mathcal U_Y)\subset\tilde{\mathfrak{E}}_{\alpha} (\mathcal U_X),\ R_{\Sigma_{b Y}}\subset R_{\Sigma_{b X}}.$$
If $\varphi$ is a surjection, then 
$$\mathfrak{E}_{\alpha'} (\mathcal U_Y)\subset\mathfrak{E}_{\alpha} (\mathcal U_X),\ R_{\Sigma_{Y}}\subset R_{\Sigma_{X}}.$$
\end{pro}

\begin{proof}  Let $\tilde\varphi: bX\to bY$ be the extension (surjection) of $\varphi$ (a uniformly continuous surjective $G$-map). For any $y\in b Y$ and $\tilde\varphi (x)=y$ the following is valid $\tilde{\alpha}'_y=\tilde\varphi\circ\tilde\alpha_x$. 

By Theorem~\ref{diagactalgcont} item IV $\tilde{\mathfrak{E}}_{\alpha} (\mathcal U_X)\in\mathbb{E}(G)$ us the least equiuniformity $\mathcal U$ on $G$ such that the orbit maps $\tilde\alpha_x: (G, \mathcal U)\to (b X, \tilde{\mathcal U}_X)$, $x\in b X$, are uniformly continuous. Hence, the orbit maps $\tilde{\alpha}'_y: (G, \tilde{\mathfrak{E}}_{\alpha} (\mathcal U_X))\to (b Y, \tilde{\mathcal U}_Y)$, $y\in b Y$, are uniformly continuous. Again applying Theorem~\ref{diagactalgcont} item IV, one obtains $\tilde{\mathfrak{E}}_{\alpha'} (\mathcal U_Y)\subset\tilde{\mathfrak{E}}_{\alpha} (\mathcal U_X)$. 

Since $\tilde\varphi$ is a $G$-map $\st_x\subset\st_{\varphi (x)}$ (if $\alpha (g, x)=x$, then $\alpha'(g, \varphi (x))=\varphi(\alpha (g, x))=\varphi (x)$). The usage of definitions of $R_{\Sigma_{b Y}}$ and $R_{\Sigma_{b X}}$  finishes the proof of the inclusion $R_{\Sigma_{b Y}}\subset R_{\Sigma_{b X}}$.

The proof of the second statement is analogous.
\end{proof}


\subsection{The maps $\mathfrak{E}_{*}$ and  $\tilde{\mathfrak{E}}_{*}$}\label{descrcomp}

For a $G$-Tychonoff space $(G, G, \cdot)$ the t.p.c. $\tau_p$ coincides with the topology of $G$ by Lemma~\ref{cointpc} and, hence,  from the considerations in \S~\ref{Elliscmp} the maps  $\mathfrak{E}_{*}:\mathbb{BU}(G)\to \mathbb{BU} (G)$, $\tilde{\mathfrak{E}}_{*}:\mathbb{BU}(G)\to \mathbb{E} (G)$ are correctly defined.

\begin{thm}\label{mathfrak} For a topological group $G$ and $\mathcal U, \mathcal U'\in \mathbb{BU}(G)$  the following hold.
\begin{itemize}
\item[{\rm (a)}]  $\mathcal U\subset \mathfrak{E}_{*}(\mathcal U)\subset\tilde{\mathfrak{E}}_{*}(\mathcal U)$.
\item[{\rm (b)}]  $\mathcal U=\mathfrak{E}_{*}(\mathcal U)$ iff for any $x\in G$ the orbit map $\cdot_x: (G, \mathcal U)\to  (G, \mathcal U)$, $\cdot_x (g)=gx$, is uniformly continuous  {\rm(}in particular,  $\mathfrak{E}_{*}(\mathfrak{E}_{*}(\mathcal U))=\mathfrak{E}_{*}(\mathcal U)${\rm)}.
\item[{\rm (c)}] If $\mathcal U\subset\mathcal U'$, then $\mathfrak{E}_{*}(\mathcal U)\subset\mathfrak{E}_{*}(\mathcal U')$ {\rm(}in particular, if $\mathcal U\subset\mathcal U'\subset\mathfrak{E}_{*}(\mathcal U)$, then $\mathfrak{E}_{*}(\mathcal U)=\mathfrak{E}_{*}(\mathcal U')${\rm)}.
\item[{\rm (d)}] If  $\mathcal U\in\mathbb{E} (G)\subset\mathbb{BU}(G)$, then  $\tilde{\mathfrak{E}}_{*}(\mathcal U)=\mathfrak{E}_{*}(\mathcal U)=\mathcal U$ {\rm(}in particular, $\tilde{\mathfrak{E}}_{*}:\mathbb{BU}(G)\to \mathbb{E} (G)$ is a surjection, $\tilde{\mathfrak{E}}_{*}(\mathcal U)=\tilde{\mathfrak{E}}_{*}(\tilde{\mathfrak{E}}_{*}(\mathcal U))$ and $\mathcal U\in\mathbb{E} (G)$ iff  $\tilde{\mathfrak{E}}_{*}(\mathcal U)=\mathcal U${\rm)}.
\item[{\rm (e)}] $\mathcal U\in \mathbb{E} (G)$ iff for any $x\in bG$ {\rm(}$(b G, \tilde{\mathcal U})$ is the completion of $(G, \mathcal U)${\rm)} the orbit map $\tilde\cdot_x:  (G, \mathcal U)\to (b G, \tilde{\mathcal U})$, $\tilde\cdot_x (g)=\tilde\cdot (g, x)$,  is uniformly continuous. 
\item[{\rm (f)}] If $\mathcal U\subset\mathcal U'$, then $\tilde{\mathfrak{E}}_{*}(\mathcal U)\subset\tilde{\mathfrak{E}}_{*}(\mathcal U')$  {\rm(}in particular, if $\mathcal U\subset\mathcal U'\subset\tilde{\mathfrak{E}}_{*}(\mathcal U)$, then $\tilde{\mathfrak{E}}_{*}(\mathcal U)=\tilde{\mathfrak{E}}_{*}(\mathcal U')$ and $\tilde{\mathfrak{E}}_{*}(\mathcal U)=\tilde{\mathfrak{E}}_{*}({\mathfrak{E}}_{*}(\mathcal U))${\rm)}.
\end{itemize}
\end{thm}

\begin{proof}
(a) The projection $\pr_{e}: (G^{G}, \mathcal U^p)\to (G=\{e\}\times G, \mathcal U)$ is uniformly continuous. Therefore,  $\pr_{e}|_{\imath (G)}: (\imath (G),  \mathfrak{E}_{*}(\mathcal U))\to (G, \mathcal U)$ is uniformly continuous. Hence,  $\mathcal U\subset \mathfrak{E}_{*}(\mathcal U)$. The inclusion $\mathfrak{E}_{*}(\mathcal U)\subset\tilde{\mathfrak{E}}_{*}(\mathcal U)$ is proved in Proposition~\ref{order}.

(b) follows from the inclusion  $\mathcal U\subset \mathfrak{E}_{*}(\mathcal U)$ in (a) and Theorem~\ref{diagactalgcont} item IV.

(c) The preservation of order follows from Proposition~\ref{order}. The rest follows from the inclusion  $\mathcal U\subset \mathfrak{E}_{*}(\mathcal U)$ and Theorem~\ref{diagactalgcont} item IV.

(d) By (a)  $\mathcal U\subset\tilde{\mathfrak{E}}_{*}(\mathcal U)$. 

Let $(b G, b)$ be the $G$-compactification of $G$ ($(b G, \tilde{\mathcal U})$ is the completion of $(G, \mathcal U)$), $\tilde\cdot: G\times bG\to bG$ is the extension of action $\cdot$, $\tilde\imath: G\to (bG)^{bG}$. 

Since $(bG, \bullet)$ is a right topological semigroup and the restriction of multiplication on $b (G)\times bG$ is continuous:
$$\imath (g)(t)=\tilde\cdot (g, t)=b (g)\bullet t,\ t\in bG,\ \mbox{and}$$
the continuous multiplication on the right is uniformly continuous. 
Therefore, for any $t\in bG$ and   for any ${\rm U}\in \mathcal U$ there exists ${\rm V}\in\mathcal U$ such that 
$$\mbox{if for}\ f, g\in G,\ (f, g)\in {\rm V},\ \mbox{then}\ (b (f)\bullet t, b (g)\bullet t)\in {\rm U}.$$
Therefore, for any entourage $\hat {\rm U}=\{ (f, g)\in G\times G\ |\ (\imath (f)(t)=b (f)\bullet t, \imath (g)(t)=b (g)\bullet t)\in {\rm U}\}$, where $t\in bG$, ${\rm U}\in\mathcal U$, from the subbase of $\tilde{\mathfrak{E}}_{*}(\mathcal U)$ there exists  ${\rm V}\in\mathcal U$ such that 
$$\mbox{if}\ (f, g)\in {\rm V},\ \mbox{then}\  (f, g)\in \hat {\rm U}$$
and the map $\tilde\imath: G\to\tilde\imath (G)$ is uniformly continuous. Hence,  $\tilde{\mathfrak{E}}_{*}(\mathcal U)\subset\mathcal U$ and, finally $\mathcal  U=\tilde{\mathfrak{E}}_{*}(\mathcal U)$. 

The inclusion $\mathbb{E} (G)\subset\mathbb{BU}(G)$ yields that $\tilde{\mathfrak{E}}_{*}$ is surjective.  Idempotentness of $\tilde{\mathfrak{E}}_{*}$ is evident.

\medskip

(e) $\tilde{\mathfrak{E}}_{*}(\mathcal U)=\mathcal U\stackrel{\mbox{\footnotesize{(Theorem~\ref{diagactalgcont} item IV)}}}\Longrightarrow$ for any $x\in bG$ {\rm(}$(b G, \tilde{\mathcal U})$ is the completion of $(G, \mathcal U)${\rm)} the orbit map $\tilde\cdot_x:  (G, \mathcal U)\to (b G, \tilde{\mathcal U})$, $\tilde\cdot_x (g)=\tilde\cdot (g, x)$,  is uniformly continuous $\stackrel{\mbox{\footnotesize{(Theorem~\ref{diagactalgcont} item IV)}}}\Longrightarrow\tilde{\mathfrak{E}}_{*}(\mathcal U)\subset\mathcal U\stackrel{(a)}\Longrightarrow\tilde{\mathfrak{E}}_{*}(\mathcal U)=\mathcal U$.

\medskip

(f) The first statement follows from  Proposition~\ref{order}.  If $\mathcal U\subset\mathcal U'\subset\tilde{\mathfrak{E}}_{*}(\mathcal U)$, then 
$$\tilde{\mathfrak{E}}_{*}(\mathcal U)\subset \tilde{\mathfrak{E}}_{*}(\mathcal U')\subset \tilde{\mathfrak{E}}_{*}(\tilde{\mathfrak{E}}_{*}(\mathcal U))\stackrel{(d)}{=}\tilde{\mathfrak{E}}_{*}(\mathcal U)$$ 
and $\tilde{\mathfrak{E}}_{*}(\mathcal U)=\tilde{\mathfrak{E}}_{*}(\mathcal U')$. Applying (a) the last statement is obtained. 
\end{proof}

\begin{ex}{\rm 

1. {\it The greatest ambit of $G$} is the pair $(\beta_G G, e)$ where $\beta_G G$ is the completion of $G$ with respect to the maximal totally bounded equiuniformity $\mathcal U_{{\rm max}}$ on $G$~\cite{Brook}.

$\mathcal U_{{\rm max}}\in \mathbb{E} (G)$, the greatest ambit $\beta_G G$ is the maximal proper enveloping Ellis semigroup compactification of $G$.

Indeed,  $\mathcal U_{{\rm max}}\subset \mathfrak{E}_{*}(\mathcal U_{{\rm max}})\subset\tilde{\mathfrak{E}}_{*}(\mathcal U_{{\rm max}})\subset\mathcal U_{{\rm max}}\in\mathbb {BU}(G)$ by Theorem~\ref{mathfrak} and the maximality of $\mathcal U_{{\rm max}}\in \mathbb{BU} (G)$. Hence,  $\mathcal U_{{\rm max}}=\tilde{\mathfrak{E}}_{*}(\mathcal U_{{\rm max}})\in\mathbb{E} (G)$.

If $R=\mathcal U_{{\rm max}}$, then $G$ is a precompact group (see, for example,~\cite[\S\ 3.7]{ArhTk}) and the unique  (see, for example,~\cite{ChK2}) proper enveloping Ellis semigroup compactification of $G$ is a topological group.

For a non-Archimedian group $G$ the spectral description of  $\beta_G G$ is given in~\cite[Corollary 3.3]{pestov1998}.

\medskip

2. Let $\mathcal U_{L\wedge R}$ be the precompact reflection of the Roelcke uniformity $L\wedge R$. $\mathcal U_{L\wedge R}={\mathfrak{E}}_{*}(\mathcal U_{L\wedge R})$.

Indeed, for any cover $\{OgO\ |\ g\in G\}\in L\wedge R$, $O\in N_G(e)$, and any $h\in G$ there exists  $V\in N_G(e)$ such that $Vh\subset hO$. Hence, the orbit map $\cdot_h: (G, L\wedge R)\to (G, L\wedge R)$, $\cdot_h(g)=gh$ is uniformly continuous ($\{VgV\ |\ g\in G\}h=\{VgVh\ |\ g\in G\}\succ\{OghO\ |\ g\in G\}=\{OgO\ |\ g\in G\}$). Therefore, for the precompact reflection $\mathcal U_{L\wedge R}$ of  $L\wedge R$  the orbit map $\cdot_h: (G, \mathcal U_{L\wedge R})\to (G, \mathcal U_{L\wedge R})$, $\cdot_h(g)=gh$ is uniformly continuous

Since $L\wedge R$ is an equiuniformity for both actions: multiplication on the left and the right, its precompact reflection $\mathcal U_{L\wedge R}$ is a totally bounded equiuniformity for the multiplication on the left and the right. Hence, the actions $\cdot_l:G\times G\to G$, $\cdot_l(g, x)=gx$ and  $\cdot_r: G\times G\to G$, $\cdot_r(g, x)=xg$ extend to the continuous actions on the completion of $(G, \mathcal U_{L\wedge R})$. 

Since the inversion $\i: G\to G$, $\i(g)=g^{-1}$ is a uniform equivalence with respect to $L\wedge R$, it is a uniform equivalence with respect to $\mathcal U_{L\wedge R}$ and, hence, extends to the uniform equivalence of the completion of $(G, \mathcal U_{L\wedge R})$. 

\medskip

From Theorem~\ref{mathfrak} one has  $\mathcal U_{L\wedge R}\in\mathbb {E}(G)$ iff $\mathcal U_{L\wedge R}=\tilde{\mathfrak{E}}_{*}(\mathcal U_{L\wedge R})$.

\medskip

3. If $G$ is a locally compact group then,  the least totally bounded uniformity $\mathcal U_{\alpha}$ on $G$ is an Ellis equiuniformity and the one-point Alexandroff compactification $\alpha G$ is a proper semitopological semigroup compactification of $G$ {\rm\cite{Berg}}.

The action $\cdot: G\times G\to G$ extends to a continuous action $\tilde\cdot: G\times\alpha G\to\alpha G$. By Theorem~\ref{mathfrak} item (e) if for any $x\in \alpha G$ the map $\tilde\cdot_x:  (G, \mathcal U_{\alpha})\to (\alpha G=G\cup\{\infty\}, \tilde{\mathcal U}_{\alpha})$, $\tilde\cdot_x (g)=\tilde\cdot(g, x)$, is uniformly continuous, then  $\mathcal U_{\alpha}\in\mathbb E(G)$ and  $\alpha G\in\mathfrak E(G)$.

If $x\in G$, then the extension $\widetilde{\tilde\cdot_x}: \alpha G\to\alpha G$, $\widetilde{\tilde\cdot_x}(\infty)=\infty$,  of the map $\tilde\cdot_x:  G\to \alpha G$ is a continuous map and, hence, uniformly continuous. Therefore its restriction to $G$ is uniformly continuous.

If $x=\infty$, then $\tilde\cdot (g, \infty)=\infty$,  the map $\tilde\cdot_{\infty}:  G\to \alpha G$ is constant and, hence, uniformly continuous. 

The extension $\tilde\imath: \alpha G\to\tilde\imath (\alpha G)\subset (\alpha G)^{\alpha G}$ of $\imath$ is a topological isomorphism of semigroups and  $\tilde\imath (\infty)(x)=\widetilde{\tilde\cdot_x} (\infty)=\infty$, $x\in\alpha G$. Hence,  $\tilde\imath (\infty)$ is a continuous map, the multiplication on the left in $\tilde\imath (\alpha G)$ is continuous, and $\alpha G$ is a proper semitopological semigroup compactification of $G$.}
\end{ex}

From Proposition~\ref{connection}, Theorem~\ref{diagactalgcont}  item IV and Theorem~\ref{mathfrak} item (a) one has.

\begin{cor}\label{connectiongroup}
Let $(G=(G, \tau_p), X, \alpha)$ be a $G$-Tychonoff space, $\mathcal U_{X}\in\mathbb{BU}(X)$ and  $(b X, \tilde{\mathcal U}_X)$ is the completion of  $(X, \mathcal U_X)$ {\rm(}$G$-compactification of $X${\rm)}, $\mathcal U\in\mathbb{BU}(G)$ and  $(b G, \tilde{\mathcal U})$ is the completion of  $(G, \mathcal U)$, $\alpha_x: (G, \mathcal U)\to  (X, \mathcal U_{X})$ is a uniformly continuous $G$-map, $x\in X$, $\alpha_x (G)$ is dense in $X$ for some $x\in X$. Then  
$\tilde{\mathfrak{E}}_{\alpha} (\mathcal U_X)\subset\tilde{\mathfrak{E}}_{\star} (\mathcal U)$ and  $\mathfrak{E}_{\alpha} (\mathcal U_X)\subset\mathcal U\subset\mathfrak{E}_{\star} (\mathcal U)$.
\end{cor}


\section{The maps $\mathfrak{E}_\alpha$ and  $\tilde{\mathfrak{E}}_\alpha$ for a uniformly equicontinuous action}\label{unifequi}

An action $\alpha: G\times (X, \mathcal L)\to  (X, \mathcal L)$ is {\it uniformly equicontinuous} if $\{\alpha^g: X\to X,\ \alpha^g(x)=\alpha (g, x)\ |\ g\in G\}$  is a uniformly equicontinuous family of maps on $X$. Equivalently, for any $u\in\mathcal L$ there exists $v\in\mathcal L$  such that $gv\succ u$, for any $g\in G$. 
If an action  $G\curvearrowright (X, \mathcal L)$ is uniformly equicontinuous, then the t.p.c. $\tau_p$ is the smallest admissible group topology on $G$~\cite[Lemma 3.1]{Kozlov}. Moreover, $((G, \tau_p), X, \alpha)$ is a $G$-Tychonoff space~\cite{Megr1984}. 

\begin{rem}{\rm Let $((G, \tau_p), X, \curvearrowright)$ be a $G$-space. If $X$ is a compactum and the action  $G\curvearrowright X$ is (uniformly)  equicontinuous, then the completion of $(G, {\mathfrak E}_{\alpha}(\mathcal U_X))$ is a compact topological group~\cite[Theorem 4.2]{Vries} (see, also~\cite[Ch.10,\ \S\ 3]{Burb} and \cite[Theorem 3.33]{Kozlov}).}
\end{rem}

Let $H$ be a neutral subgroup of a topological group $G$  (see, \cite[Definition 5.29]{RD}) without invariant subgroups. Then $G$ is a transitive uniformly equicontinuous group of homeomorphisms of $(G/H, \mathcal L)$, where $\mathcal L$ is the quotient uniformity on $G/H$ of the left uniformity $L$ on $G$ (the final uniformity with respect to the map $(G, L)\to G/H$). By~\cite[Theorem 5.28]{RD} $H$ is a neutral subgroup of $(G, \tau_p)$ and $((G, \tau_p), G/H, \alpha)$ is a $G$-Tychonoff space (with an {\it open action})~\cite[Proposition 3.2]{Kozlov}.  

\begin{thm}\label{autultr} Let $H$ be a neutral subgroup without nontrivial invariant subgroups of a topological group $G$, $\mathcal U_{G/H}$ is the maximal equiuniformity on $G/H$ for the action $\alpha: (G, \tau_p)\times G/H\to G/H$. If $\mathcal U_{G/H}\in\mathbb{BU} (G/H)$,  then 
$$\begin{array}{ccclc}
{\mathfrak E}_{\alpha} (\mathcal U_{G/H}) & \subset & L\wedge R & \subset &  R_{\Sigma_{G/H}} \\
\cap & & &  & \cap \\
\widetilde{{\mathfrak E}}_{\alpha}(\mathcal U_{G/H}) &  & \subset &  & R_{\Sigma_{\beta_G {G/H}}}, 
\end{array}\leqno{(1)}$$
{\rm (2)}  $\widetilde{{\mathfrak E}}_{\alpha}(\mathcal U_{G/H})\subset\widetilde{{\mathfrak E}}_{\star}(\mathcal U_{L\wedge R})$, where $\mathcal U_{L\wedge R}$ is the precompact reflection of $L\wedge R$,

\medskip

\noindent {\rm (3)} if $L\wedge R\in\mathbb{BU}(G)$ and $L\wedge R\subset\widetilde{{\mathfrak E}}_{\alpha}(\mathcal U_{G/H})$, then  $\widetilde{{\mathfrak E}}_{\alpha}(\mathcal U_{G/H})=\widetilde{{\mathfrak E}}_{\star}(L\wedge R)$, 

\medskip

\noindent {\rm (4)} if the condition $(\star\star)$ is valid for $(G, (b {G/H}, \tilde{\mathcal U}_{G/H}), \tilde\alpha)$, then  the condition  $(\star\star)$ is valid for $(G, ({G/H}, \mathcal U_{G/H}), \alpha)$ and ${\mathfrak E}_{\alpha} (\mathcal U_{G/H})= L\wedge R=R_{\Sigma_{G/H}}$, $\widetilde{{\mathfrak E}}_{\alpha}(\mathcal U_{G/H})=R_{\Sigma_{\beta_G {G/H}}}$, $\widetilde{{\mathfrak E}}_{\alpha}(\mathcal U_{G/H})=\widetilde{{\mathfrak E}}_{\star}(L\wedge R)$.
\end{thm}

\begin{proof} (1) The inclusions ${\mathfrak E}_{\alpha}(\mathcal U_{G/H})\subset \widetilde{{\mathfrak E}}_\alpha(\mathcal U_{G/H})$, $R_{\Sigma_{G/H}}\subset R_{\Sigma_{\beta_G {G/H}}}$ and $\widetilde{{\mathfrak E}}_{\alpha}(\mathcal U_{G/H})\subset R_{\Sigma_{\beta_G {G/H}}}$ follow from Corollary~\ref{cora2-2}. The inclusion  $L\wedge R\subset R_{\Sigma_{G/H}}$ follows from Remark~\ref{a2-1}. 

By Proposition 3.25~\cite{Kozlov} $\mathcal U_{G/H}=\mathcal L\wedge\mathcal R$, where $\mathcal L\wedge\mathcal R$ is the quotient uniformity of $L\wedge R$ on $G$ for the quotient map $q: G\to G/H$ ($\mathcal U_{G/H}$ is the quotient uniformity of the right uniformity $R$ on $G$). 

The inclusion ${\mathfrak E}_{\alpha}(\mathcal U_{G/H})\subset L\wedge R$ follows from the uniform continuity of the surjective orbit maps $\alpha_x: (G, L\wedge R)\to (G/H, \mathcal L\wedge\mathcal R)$, $x\in G/H$, by Corollary~\ref{connectiongroup}.

\medskip

(2) The surjective $G$-map $q: (G,  L\wedge R)\to (G/H, \mathcal L\wedge\mathcal R)$ is uniformly continuous. Since $\mathcal L\wedge\mathcal R\in\mathbb{BU}({G/H})$, for the precompact reflection $\mathcal U_{L\wedge R}$ of $L\wedge R$ the map  $q: (G, \mathcal U_{L\wedge R})\to (G/H, \mathcal L\wedge\mathcal R)$ is uniformly continuous. By Proposition~\ref{connection} $\widetilde{{\mathfrak E}}_{\alpha}(\mathcal U_{G/H})\subset\widetilde{{\mathfrak E}}_{\star}(\mathcal U_{L\wedge R})$.

\medskip

(3) $L\wedge R=\mathcal U_{L\wedge R}$. From the inclusions $L\wedge R\subset\widetilde{{\mathfrak E}}_{\alpha}(\mathcal U_{G/H})\subset\widetilde{{\mathfrak E}}_{\star}(L\wedge R)$ (item (2)) by items (d) and  (f) of Theorem~\ref{mathfrak}  $\widetilde{{\mathfrak E}}_{\alpha}(\mathcal U_{G/H})=\widetilde{{\mathfrak E}}_{\star}(L\wedge R)$, since  $\widetilde{{\mathfrak E}}_{\alpha}(\mathcal U_{G/H})\in\mathbb E (G)$.

\medskip

(4) If the condition $(\star\star)$ is valid for $(G, (b G/H, \tilde{\mathcal U}_{G/H}), \tilde\alpha)$, then  the condition  $(\star\star)$ is valid for $(G, (G/H, \mathcal U_{G/H}), \alpha)$ by Corollary~\ref{cora2-2} item (3). The rest follows from items (1), (3) and Proposition~\ref{a2-2}. 
\end{proof}


\subsection{The maps $\mathfrak{E}_\alpha$ and  $\tilde{\mathfrak{E}}_\alpha$ for an action on a discrete space}

Let $G$ be a subgroup of the permutation group ${\rm S}(X)$  of a discrete (infinite) space $X$. A discrete space $X$ can be considered as a metric space with distance $1$ between distinct points. Then  ${\rm S}(X)$ is an isometry group and $G$ is its subgroup. The permutation topology $\tau_{\partial}$ on $G$ is the smallest admissible group topology (and coincides with the topology of pointwise convergence) for the action $\alpha: G\times X\to X$. The base of the nbds of the unit of $(G, \tau_{\partial})$ is formed by clopen subgroups 
$$\st_{\sigma}=\{g\in G\ |\ \alpha (g, x)=x,\ x\in\sigma\},\  \sigma\in\Sigma_X,$$
and  $(G, \tau_{\partial})$ is non-Archimedean.

\begin{lem}\label{coralpha}
Let $X$ be discrete, $G$ is a subgroup of  ${\rm S}(X)$. Then for $(G,  \tau_{\partial})$ 
$$L\wedge R=R_{\Sigma_X}.$$
\end{lem}

\begin{proof}
The base of the Roelcke uniformity $L\wedge R$ on  $(G, \tau_{\partial})$ is formed by the covers
$$\{\st_{\sigma}g\st_{\sigma}\ |\ g\in G\},\ \sigma\in\Sigma_X,$$
which form the base of  the uniformity $R_{\Sigma_X}$ defined by the family of stabilizers $\st_{\sigma}$, $\sigma\in\Sigma_X$ (see \S~\ref{uniftopgr}). Hence, $L\wedge R=R_{\Sigma_X}$.
\end{proof}


\subsubsection{The maps $\mathfrak{E}_\alpha$ and  $\tilde{\mathfrak{E}}_\alpha$ for an ultratransitive action on a discrete space}\label{discrete}

\begin{thm}\label{Roelcke precomp4-1} Let $X$ be a discrete space and $\mathcal U_X$ is the maximal equiuniformity on $X$ for the action $\alpha:({\rm S}(X), \tau_{\partial})\times X\to X$.  Then 
\begin{itemize}
\item[{\rm (a)}] $\mathbb{BU}(X)=\{\mathcal U_X\}$ and $\tilde X=\alpha X$, where $(\tilde X, \tilde{\mathcal U}_X)$ is the completion of $(X, \mathcal U_X)$, $\alpha X=X\cup\{\infty\}$ is the Alexandroff one-point compactification of $X$, 
\item[{\rm (b)}]  all inclusions in {\rm Theorem~\ref{autultr}} item {\rm(1)} are equalities and $\tilde{\mathfrak{E}}_{\alpha}(\mathcal U_X)=L\wedge R$ {\rm(}hence, $({\rm S}(X), \tau_{\partial})$ is Roelcke precompact and the Roelcke compactification $(b_r {\rm S}(X), b_r)$ of $({\rm S}(X), \tau_{\partial})$ is  a proper enveloping Ellis semigroup compactification{\rm)}, 
\item[{\rm (c)}]  $b_r {\rm S}(X)$ is the set of selfmaps $f$ of $\alpha X$  in the t.p.c.  such that $f(\infty)=\infty$, $f$ is a bijection on $Y\subset X$ and $f(X\setminus Y)=\infty$, where $Y$ is an arbitrary subset of $X$, 
\item[{\rm (d)}] $b_r {\rm S}(X)$ is a semitopological semigroup.
\end{itemize}
\end{thm}

\begin{proof} (a) The base of the maximal equiuniformity $\mathcal U_X$ on $X$ is formed by the covers 
$$u_{\theta}=\{\st_{\theta}x=\alpha (\st_{\theta}, x)\ |\ x\in X\},\ \theta\in\Sigma_X,\ {\rm(}u_{\theta'}\succ u_{\theta}\ \mbox{if}\ \theta\subset\theta'{\rm )},$$
see, for example, \cite{ChK}. Since $\st_{\theta}x=x$ if $x\in\theta$ and  $\st_{\theta}x=\alpha X\setminus\theta$ if $x\notin\theta$ they are of the  form 
$$\{\{\{x\}\ |\ {x\in\theta}\},\ X\setminus\theta\},\ \theta\in\Sigma_X,$$
and coincides with the restriction on $X$ of the unique uniformity on $\alpha X$. Hence $\alpha X=\tilde X$ and the maximal equiuniformity $\mathcal U_X$ on $X$ is totally bounded and unique.

(b) From Lemma~\ref{coralpha} $L\wedge R=R_{\Sigma_X}$. Since  $\st_{\infty}={\rm S}(X)$ for the extended action $\tilde\alpha: {\rm S}(X)\times\alpha X\to\alpha X$ (continuity of each $g\in {\rm S}(X)$ yields that $\tilde\alpha (g, \infty)=\infty$),  $R_{\Sigma_X}=R_{\Sigma_{\alpha X}}$. Since $\tilde\alpha_{\infty}(g)=\tilde\alpha (g, \infty)=\infty$, $g\in {\rm S}(X)$, the restriction to $\tilde\imath ({\rm S}(X))$ of the projection $\pr_{X}: (\alpha X)^{\alpha X}\to (\alpha X)^X$ is a topological isomorphism and $\mathfrak{E}_{\alpha}(\mathcal U_X)=\widetilde{\mathfrak{E}}_{\alpha}(\mathcal U_X)$.

To end the proof of (b) it is enough to check the inclusion $\mathfrak{E}_{\alpha}(\mathcal U_X)\supset R_{\Sigma_X}$. In the case of the equiuniformity  $\mathcal U_X$ (after identification of coinciding sets, the covers from the base consist of pairwise disjoint sets) the condition $(\star\star)$ in Proposition~\ref{a2-2} can be reformulated in the following form in terms of covers $u_{\sigma}\in\mathcal U_X$, 
$$\begin{array}{c}
\mbox{\rm if for}\ g, h\in {\rm S}(X)\ \ \alpha(h, x)\in\st_{\theta}\alpha(g, x)\ \mbox{\rm holds for}\ x\in\sigma,\ \mbox{\rm then}\\
\mbox{\rm there exists}\ \ g'\in g\st_{\sigma}\ \ \mbox{\rm such that}\ \ h\in\st_{\theta}g',\ \mbox{where, without loss of generality,}\ \sigma\subset\theta\in\Sigma_X.\\ 
\end{array}$$
For any $g, h\in {\rm S}(X)$ such that $\alpha(h, x)\in\st_{\theta}\alpha(g, x)$, $x\in\sigma$, we consider, without loss of generality, the following. 
 
Put $\sigma'=\{x\in\sigma\ |\ \alpha (g, x)\in\theta\}\subset\sigma$. Then $\alpha (h, x)=\alpha (g, x)\in\theta$, $x\in\sigma'$. If $\sigma'\ne\sigma$, then $\alpha (g, x)\not\in\theta$,  $\alpha (h, x)\not\in\theta$ for  $x\in\sigma\setminus\sigma'$.

If $\sigma'=\sigma$, then $hg^{-1}\in\st_{\theta}$ and $h\in\st_{\theta}g$, $g\in g\st_{\sigma}$.

\medskip 
 
If $\sigma'\ne\sigma$, then there exists $f\in {\rm S}(X)$ such that the points $x\in\theta, \alpha (g, y)$, where  $y\in\sigma\setminus\sigma'$, are mapped respectively to the points $x\in\theta, \alpha (h, y), y\in\sigma\setminus\sigma'$ (for fixed orders on $\theta$ and $\sigma\setminus\sigma'$).  $f\in\st_{\theta}$ and $\alpha (h, x)=\alpha (fg, x),\ x\in\sigma$. 
Hence, $h^{-1}f g\in\st_{\sigma}$. Consequently $f g\in h\st_{\sigma}$ (equivalently $h\in (fg)\st_{\sigma}$) and there is $g'\in g\st_{\sigma}$ such that $h\in \st_{\theta}g'$. The  inclusion $\mathfrak{E}_{\alpha}(\mathcal U_X)\supset R_{\Sigma_X}$ and, hence, the equality $\mathfrak{E}_{\alpha}(\mathcal U_X)=R_{\Sigma_X}$ are proved.

\medskip

(c) At first, let us note that for any $f\in\cl~\tilde\imath ({\rm S}(X))$ there are no points $x\ne y\in X$ such that $f(x)=f(y)\in X$. Indeed, suppose that there are $x\ne y\in X$ such that $f(x)=f(y)\in X$. One can take the nbd $W=\{h\in\cl~\tilde\imath ({\rm S}(X))\ |\ h(x)=h(y)=f(x)\}$ of $f$. Then $W\cap\tilde\imath ({\rm S}(X))=\emptyset$ and $f\not\in\cl\tilde\imath ({\rm S}(X))$. 

Secondly, $f(\infty)=\infty$ for any $f\in\cl~\tilde\imath ({\rm S}(X))$. Indeed, if $f(\infty)=x\in X$, then take the nbd $W=\{h\in\cl~\tilde\imath ({\rm S}(X))\ |\ h(\infty)=x\}$ of $f$ and $W\cap\tilde\imath ({\rm S}(X))=\emptyset$ since  $g(\infty)=\infty$ for any $g\in\tilde\imath ({\rm S}(X))$.

It remains to check that $f\in b_r {\rm S}(X)$ if $f(\infty)=\infty$, $f$ is a bijection on $Y\subset X$ and $f(X\setminus Y)=\infty$, where $Y$ is an arbitrary subset of $X$. Any nbd of $f$ is of the form $\{g\in (\alpha X)^{\alpha X}\ |\ g(x)\in Of(x),\ x\in\sigma\}\cap b_r {\rm S}(X)$, $\sigma\in\Sigma_X$, where $Of(x)=f(x)$ if $f(x)\ne\infty$, and $Of(x)=\alpha X\setminus\sigma_x$, $\sigma_x\in\Sigma_X$, if $f(x)=\infty$. By finitness of $\sigma$ and $\sigma_x$, $x\in\sigma$ there is $h\in {\rm S}(X)$ such that $h(x)=f(x)$ if $x\in\sigma$ and $f(x)\in X$, and $h(x)\not\in\sigma_x$, if $x\in\sigma$ and $f(x)=\infty$. Evidently, $h\in\{g\in (\alpha X)^{\alpha X}\ |\ g(x)\in Of(x),\ x\in\sigma\}\cap b_r {\rm S}(X)$ and $f\in b_r {\rm S}(X)$.

\medskip

(d) From (b) $b_r {\rm S}(X)=\cl~\tilde\imath ({\rm S}(X))$ is a right topological semigroup, where $\tilde\imath$ is a topological isomorphism of ${\rm S}(X)$ into $(\alpha X)^{\alpha X}$, see \S~\ref{digonal}. If one shows that all elements of $\cl\tilde\imath ({\rm S}(X))$ are continuous (as selfmaps of $\alpha X$), then multiplication on the left is continuous. To check the continuity of $f\in\cl~\tilde\imath ({\rm S}(X))$ it is sufficient to check the continuity of $f$ at the point $\infty$.

Suppose that $f\in\cl~\tilde\imath ({\rm S}(X))$ is not continuous at the point $\infty$. Then there exists a nbd $W=\alpha X\setminus\sigma$ of $\infty$ for some $\sigma\in\Sigma_X$  such that the set $\{x\in X\ |\ f(x)\in\sigma\}$ is infinite.  Since $\sigma$ is finite, there exist $x\ne y\in X$ such that $f(x)=f(y)$. Therefore,  $f\not\in\cl~\tilde\imath ({\rm S}(X))$. The obtained contradiction shows that all  $f\in\cl~\tilde\imath ({\rm S}(X))$ are continuous and  $b_r {\rm S}(X)$ is a semitopological semigroup.
\end{proof}

\begin{rem}
{\rm (1)  $b_r  {\rm S}(X)$ is the set of maps $f: X\to\alpha X$  in the t.p.c. such that $f$ is a bijection on $Y\subset X$ and $f(X\setminus Y)=\infty$, where $Y$ is an arbitrary subset of $X$.

(2) In~\cite[\S\ 12, Theorem 12.2]{GlasnerMegr2008} the description of $b_r  {\rm S}(\mathbb N)$ is given. It is also shown that $b_r  {\rm S}(\mathbb N)$ is a semitopological semigroup homeomorphic to the Cantor set.}
\end{rem}

\begin{df} An action of a group $G$ on a set $X$ is strongly $n$-transitive, $n\geq 1$, if for any families of distinct $n$ points $x_1, \ldots, x_n$ and $y_1, \ldots, y_n$ there exists $g\in G$ such that $g(x_k)=y_k$, $k=1,\ldots, n$. 

An action $G\curvearrowright X$, which is strongly $n$-transitive for all $n\in\mathbb N$, is called ultratransitive.
\end{df}

\begin{rem}\label{remultr}{\rm 
(1) A group $G$ which acts ultratransitively on $X$ is a dense subgroup of $({\rm S} (X), \tau_{\partial})$. 

Indeed, let a set $O$ be open in $({\rm S}(X), \tau_{\partial})$ and $g\in O$. Then there are $\sigma\in\Sigma_X$ such that $g\st_{\sigma}\subset O$ and $g\st_{\sigma}=\{h\in G\ |\ h(x)=g(x),\ x\in\sigma\}$ is an open neighbourhood of $g$. Since $G$ acts ultratransitively on $X$, there is $h\in G$ such that $h(x)=g(x)$, $x\in\sigma$. Evidently, $h\in O$.

(2) The subgroup ${\rm S}_{<\omega}(X)$ of the group ${\rm S}(X)$ whose elements has finite supports acts ultratransitively on $X$. Hence, ${\rm S}_{<\omega} (X)$ is a dense subgroup of $({\rm S}(X), \tau_{\partial})$. The group whose elements have finite supports is a subgroup of any Houghton's group (see, for example, \cite{Cox}). Hence, Houghton's groups act ultratransitively on the corresponding countable discrete spaces.

(3)  Any group $G$ which acts ultratransitively on $X$ is {\it oligomorphic} and, hence,  Roelcke precompact (see, \cite{Tsan}, \cite{Sorin1}). The Roelcke precompactness of $({\rm S}(X), \tau_{\partial})$ is proved in~\cite{Gau}  (see also~\cite[Example 9.14]{RD}), 
the Roelcke precompactness of $({\rm S}_{<\omega}(X), \tau_{\partial})$ is proved in~\cite{Ban}. 

From (1) it follows that the Roelcke compactification of a group $G$ which acts ultratransitively on $X$ is isomorphic to the Roelcke compactification of  $({\rm S}(X), \tau_{\partial})$.

(4) A space $X$ is {\it ultrahomogeneous} if the action of its homeomorphism group $\Hom(X)$ is ultratransitive. Locally compact metrizable CDH spaces whose complement to any finite subset is connected are ultrahomogeneous spaces. They include the spheres $S^{n-1}$ in Euclidean spaces $\mathbb R^n$, $n\geq 3$, the Hilbert cube $Q$ (see, for example,~\cite{ArM}). 

The group of homeomorphisms of an ultrahomogeneous space with the permutation topology is Roelcke precompact, although the permutation topology is not necessarily admissible. The homeomorphism groups of the spheres $S^{n}$, $n\geq 2$, and the Hilbert cube $Q$ in the compact open topology (the smallest admissible group topology) are not Roelcke precompact~\cite{Ros}. Hence, in any admissible group topology, these groups are not Roelcke precompact.}
\end{rem}

\begin{cor}\label{cordis}
Let $G$ acts ultratransitively on a discrete space $X$. Then for the action $(G, \tau_{\partial})\curvearrowright X$ items {\rm (a)} and {\rm (b)} of {\rm Theorem~\ref{Roelcke precomp4-1}} hold and $b_r (G, \tau_{\partial})$ and $b_r {\rm S}(X)$ are topologically isomorphic semitopological semigroups. 
\end{cor}

\begin{proof}
The same reasonings as in the proof of Theorem~\ref{Roelcke precomp4-1} work in the case of ultratransitive action to prove analogs of items (a) and (b). 

$(G, \tau_{\partial})$ is a dense subgroup of $({\rm S}(X), \tau_{\partial})$ by item (1) of Remark~\ref{remultr} and is Roelcke precompact~\cite[Proposition 3.24]{RD}. Moreover, these Roelcke compacifications of $(G, \tau_{\partial})$ and  $({\rm S}(X), \tau_{\partial})$ are uniformly isomorphic (use~\cite[Proposition 3.24]{RD} and~\cite[Corollary 8.3.11]{Engelking}) by the extension of the identity map on $(G, \tau_{\partial})$. Hence,  $b_r (G, \tau_{\partial})$ is a semitopological semigroup topologically isomorphic to  $b_r S(X)$.
\end{proof}


\subsubsection{The maps $\mathfrak{E}_\alpha$ and  $\tilde{\mathfrak{E}}_\alpha$ for an ultratransitive action on a chain}\label{chain}

For an infinite {\it chain} ({\it linearly ordered set}) $X$ let $\aut(X)$ be the group of {\it automorphisms} (order-preserving bijections) of  $X$. A chain $X$ is called {\it homogeneous} if the action $\aut(X)\curvearrowright X$ is transitive {\rm(}for any $x, y\in X$ there exists $f\in\aut(X)$ such that $f(x)=y${\rm)}. 
A chain $X$ is called {\rm 2}-{\it homogeneous} if for any pairs of points  $x<y$ and $x'<y'$ there exists $g\in\aut(X)$ such that $g(x)=x'$, $g(y)=y'$. 

\begin{rem}{\rm 
(1) A 2-homogeneous chain is {\it dense}~\cite[Ch.\ 6, \S\ 2]{KM} (and, hence, has no {\it jumps}). 

(2) A 2-homogeneous chain is {\it ultrahomogeneous}~{\rm\cite[Lemma 1.10.1]{Glass} (see also~\cite{Ovch})}, i.e. for any families of different $n$ points $x_1<\ldots< x_n$ and   $y_1<\ldots< y_n$ there exists $g\in\aut(X)$ such that $g(x_k)=y_k$, $k=1,\ldots, n$, $n\in\mathbb N$.

(3) A chain without {\it proper gaps} is called {\it continuous}~{\rm\cite[Ch.\ 6, \S\ 2]{KM}}. If a chain $X$ is ultrahomogeneous and has a proper gap, then there is a gap on every nonempty interval of $X$ ($J\subset X$ is an {\it interval} if $x<y\in J\Longrightarrow z\in J\ \forall\ x<z<y$).

(4) On the group $\aut(X)$, the permutation topology $\tau_{\partial}$ is the smallest admissible group topology for the action  $\aut(X)\curvearrowright X$,  where $X$ is a discrete space. $(\aut(X), \tau_{\partial})$ is a subgroup $({\rm S}(X), \tau_{\partial})$. }
\end{rem}

A chain $X$  in the discrete topology is a GO-space ({\it generalized ordered space}). There is the smallest LOTS 
$$X\otimes_{\ell}\{-1, 0, 1\}\ (\mbox{topology induced by the lexicographic order on}\ X\times\{-1, 0, 1\}),$$ 
in which $X=X\times\{0\}$ is a dense subspace~\cite{Miwa}, and $X\otimes_{\ell}\{-1, 0, 1\}$ is naturally embedded in any other linearly ordered extension of $X$, in which $X$ is dense. Hence, any linearly ordered compactification of $X$ is a linearly ordered compactification of $X\otimes_{\ell}\{-1, 0, 1\}$. 
According to the description in~\cite{Fed}, the smallest linearly ordered compactification~\cite{Kauf} $b_m X$ of $X\otimes_{\ell}\{-1, 0, 1\}$ (and, hence, of $X$) is generated by replacing each gap (also improper) by a point with a natural continuation of the order (gaps in $X\otimes_{\ell}\{-1, 0, 1\}$ and  $X$ are naturally identified). Therefore, $b_m X=X^+\cup X\cup X^-\cup\Gamma$, where $X^+=X\times\{1\}$, $X^0=X\times\{0\}$, $X^-=X\times\{-1\}$, $\Gamma$ is the set of gaps.

\begin{rem} 
{\rm There is the unique linearly ordered compactifications of $X$, when $X$ is a continuous chain. It is obtained by addition  points $\sup$ and $\inf$ to $X\otimes_{\ell}\{-1, 0, 1\}$.}
\end{rem}

\begin{thm}\label{Roelcke precomp4-2-1}  Let $X$ be an  ultrahomogeneous chain and a discrete space, $G=(\aut(X), \tau_{\partial})$,  $\mathcal U_X$ is the maximal equiuniformity on $X$ for the action $\alpha:G\times X\to X$. Then 
\begin{itemize}
\item[{\rm (a)}]  $\mathcal U_X\in\mathbb{BU}(X)$ and $\tilde X=b_m  X$, where $(\tilde X, \tilde{\mathcal U}_X)$ is the completion of $(X, \mathcal U_X)$,
\item[{\rm (b)}]   $\mathfrak{E}_{\alpha}(\mathcal U_X)=L\wedge R=R_{\Sigma_X}\subset R_{\Sigma_{b_m X}}$  {\rm(}and, hence, $G$ is Roelcke precompact{\rm)}, 
\item[{\rm (c)}]  $L\wedge R\subset\widetilde{\mathfrak{E}}_{\alpha}(\mathcal U_X)= R_{\Sigma_{b_m X}}$  {\rm(}and, hence,  $\widetilde{{\mathfrak E}}_{\star}(L\wedge R)=\widetilde{{\mathfrak E}}_{\alpha}(\mathcal U_X)$ and $L\wedge R\not\in\mathbb{E}(G)$ if $R_{\Sigma_X}\ne R_{\Sigma_{b_m X}}${\rm)}, 
\item[{\rm (d)}]  the proper Ellis semigroup compactification $b G$ {\rm (}the completion of $(G,  \widetilde{{\mathfrak E}}_{\star}(L\wedge R)=R_{\Sigma_{b_m X}})${\rm )}  is the set of selfmaps $f$ of $b_m X$  in the t.p.c.  such that  
\begin{itemize}
\item[{\rm (i)}] $f$ is monotone {\rm (}if $x<y$, then $f(x)\leq f(y)${\rm)}, 
\item[{\rm (ii)}] $f(x)\not\in X$ if $x\in b_m X\setminus X$, 
\item[{\rm (iii)}] if $f(x)=f(y)$,  $x\ne y$, then $f(x)\in b_m X\setminus X$, 
\item[{\rm (iv)}] either $f((x, -1))=f((x, 0))=f((x, 1))\not\in X$, or $f((x, -1))=(y, -1)$, $f((x, 0))=(y, 0)$,  $f((x, 1))=(y, 1)$  and 
\item[{\rm (v)}] $f(\inf)=\inf$,  $f(\sup)=\sup$, 
\end{itemize}
\item[{\rm (e)}]  the Roelcke compactification $b_r G$ is  the set of maps $f$ of $X$ to $b_m X$  in the t.p.c.  such that  {\rm (i')} $f$ is monotone,  {\rm (ii')} if $f(x)=f(y)$,  $x\ne y$, then $f(x)\in b_m X\setminus X$. 
\end{itemize}
If $X$ is continuous, then 
\begin{itemize}
\item[{\rm (b')}]  $\widetilde{\mathfrak{E}}_{\alpha}(\mathcal U_X)=\mathfrak{E}_{\alpha}(\mathcal U_X)=L\wedge R=R_{\Sigma_X}=R_{\Sigma_{b_m X}}$, 
\item[{\rm (c')}] $b_r G=b G$ is a proper enveloping Ellis semigroup compactification.
\end{itemize}
\end{thm}

\begin{proof}  (a)  Firstly, the base of the maximal equiuniformity $\mathcal U_X$  for the action  $(G, \tau_{\partial})\curvearrowright X$  is formed by the finite   covers (of disjoint sets)
$$u_{\theta}=\{\st_{\theta}x\ |\ x\in X\}=(\gets, x_1)\cup\{x_1\}\cup (x_1, x_2)\cup\{x_2\}\cup\ldots\cup(x_{n-1}, x_n)\cup\{x_n\}\cup (x_n, \to),\ \theta\in\Sigma_X,$$ ($\theta=\{x_1,\ldots, x_n\}$, $x_1<\ldots<x_n\in X$,  $u_{\theta'}\succ u_{\theta}\ \mbox{if}\ \theta\subset\theta'$). Hence, $\mathcal U_X\in\mathbb{BU}(X)$.

Secondly, the smallest linearly ordered compactification $b_m X$ of $X$ is zero-dimensional.  In any open cover of a zero-dimensional compactum  $b_m X$ it is possible to refine a finite cover of pairwise distinct clopen intervals no none is a subset of the other. By introducing the order on the intervals ($[a_1, b_1]< [a_2, b_2]$ if $a_1< a_2$), we obtain a disjoint cover of the clopen intervals 
$J_1=(\gets, y_1),\ldots, J_k=(y_k, \to)$, subtracting sequentially from the $i$ interval the union of the previous ones. Let us correct the latter cover by removing endpoints of the form $(x, 0)$ from the intervals (if any), and adding them as single-point clopen intervals to the corrected interval system to obtain an open cover $\omega$ of the space $b_m X$. Then $\omega\wedge X$ (the trace of $\omega$ on $X$) obviously coincides with the cover from $\mathcal U_X$. Since any cover of $\mathcal U_X$ is extended to the cover of  $b_m X$ ($u_{\theta}$ is a trace of the cover $\omega_{\theta}=[\inf, (x_1, -1)]\cup\{(x_1, 0)\}\cup [(x_1, 1), (x_2, -1)]\cup\{(x_2, 0)\}\cup\ldots\cup[(x_{n-1}, 1), (x_n, -1)]\cup\{(x_n, 0)\}\cup [(x_n, 1), \sup]$), then  $\tilde X=b_m X$ and the action $\alpha: G\times X\to X$ is extended to the action $\tilde\alpha: G\times b_m X\to b_m X$ (for  $g\in G$ 
\begin{itemize}
\item[] $\tilde\alpha (g, (x, j))=(\alpha (g, x), j)$, $x\in X$, $j=-1, 0, 1$;
\item[] $\tilde\alpha (g, (A, B))=(\alpha (g, A), \alpha (g, B))$, $(A, B)\in\Gamma$).
\end{itemize}
It is easy to check that the action  $\tilde\alpha$ is correctly defined, it is continuous, its restriction to $X=X^0$ coincides with $\alpha$ and $X^+$, $X^0$, $X^-$,  $\Gamma$ are invariant subsets. If $\Omega_{\theta}$ is an entourage correspondent to the cover  $\omega_{\theta}$, then  $\Omega_{\theta}\circ\Omega_{\theta}=\Omega_{\theta}$. 

\medskip

The equality $L\wedge R=R_{\Sigma_X}$ and the inclusion $R_{\Sigma_X}\subset R_{\Sigma_{b_m X}}$ in (b) follow from Lemma~\ref{coralpha} and  Theorem~\ref{autultr} respectively. 

Proof of the equality $\widetilde{\mathfrak{E}}_{\alpha}(\mathcal U_X)= R_{\Sigma_{b_m X}}$ in (c). By Proposition~\ref{lemincl} $\widetilde{\mathfrak{E}}_{\alpha}(\mathcal U_X)\subset R_{\Sigma_{b_m X}}$. 

Since $\st_{(x, j)}$, $j=-1, 0, 1$, for the action $\tilde\alpha$ coincides with $\st_x$  for the action $\alpha$ and $\st_{\inf}=\st_{\sup}=G$, one can take arbitrary $\sigma\in\Sigma_{b_m X}$, where $\sigma=\sigma_X\cup\sigma_{\Gamma}$, $\sigma_X\in\Sigma_{X^0}(=\Sigma_{X})$, $\sigma_{\Gamma}\in\Sigma_{\Gamma}\setminus\{\inf, \sup\}$. 

In the case of the equiuniformity  $\tilde{\mathcal U}_X$  the condition $(\star\star)$ in Proposition~\ref{a2-2}  can be reformulated in the following form
$$\begin{array}{c}
\mbox{\rm if for}\ g, h\in G\ \ (\tilde\alpha(h, x), \tilde\alpha(g, x))\in\Omega_{\theta}\ \mbox{\rm holds for}\ x\in\sigma=\sigma_X\cup\sigma_{\Gamma}, \\
\mbox{\rm then there exists}\ \ g'\in g\st_{\sigma}\ \ \mbox{\rm such that}\ (\tilde\alpha(h, x), \tilde\alpha(g', x))\in\Omega_{\theta}\  \mbox{\rm holds for}\ x\in b_m X,\\ \mbox{where, without loss of generality,}\ \sigma_X\subset\theta\in\Sigma_{X}.\\ 
\end{array}$$

Let $\sigma=\{y_1,\ldots, y_n\}$ and $y_1<y_2<\ldots<y_n$. The construction of the automorphism $g'$ is carried out on the intervals $[\inf, y_1], [y_1, y_2],\ldots, [y_{n-1}, y_n], [y_n, \sup]$.  

(I) The map $g'$ on the interval $[\inf, y_1]$. If $(\inf, \tilde\alpha (h, y_1))\in\Omega_{\theta}$ or $\tilde\alpha (h, y_1)\in\theta$ (and hence,  $\tilde\alpha (g, y_1)=\tilde\alpha (h, y_1)$, $y_1=(x_1, 0)$, $\tilde\alpha (g, (x_1, -1))=\tilde\alpha (h, (x_1, -1))$) and  $(\inf, \tilde\alpha (h, (x_1, -1)))\in\Omega_{\theta}$, then 
$$\tilde\alpha (g', y)=\tilde\alpha (g, y)\ \mbox{for}\ y\in [\inf, y_1].$$

Otherwise, in the case $\tilde\alpha (g, y_1)\leq\tilde\alpha (h, y_1)$  if $\tilde\alpha (g, y_1)\not\in\theta$, then take  $a_1<\tilde\alpha (g, y_1)$ such that $a_1\in X^0$,  $(a_1, \tilde\alpha (g, y_1))\in\Omega_{\theta}$  (and hence, $(a_1, \tilde\alpha (h, y_1))\in\Omega_{\theta}$). If $\tilde\alpha (g, y_1)\in\theta$, then $y_1=(x_1, 0)$, $\tilde\alpha (g, y_1)=\tilde\alpha (h, y_1)$ and  $\tilde\alpha (g, (x_1, -1))=\tilde\alpha (h, (x_1, -1))$.  Take  $a_1<\tilde\alpha (g, (x_1, -1))$ such that $a_1\in X^0$,  $(a_1, \tilde\alpha (g, (x_1, -1)))\in\Omega_{\theta}$   (and hence,  $(a_1, \tilde\alpha (h, (x_1, -1)))\in\Omega_{\theta}$). Put $t=y_1$ if $\tilde\alpha (g, y_1)\not\in\theta$ and $t=(x_1, -1)$ if $\tilde\alpha (g, y_1)\in\theta$.

We assume that $y_1^-=\alpha (h^{-1}, a_1)$ and take $y_1^-<z_1^-<t$, $z_1^-\in X^0$, $a_1<\alpha (g, z_1^-)$ (using continuity of $g$ and the inequality $a_1<\tilde\alpha (g, t)$). Since $X$ is ultrahomogeneous, there exists $\varphi_1\in G$ such that $\alpha (\varphi_1, y_1^-)=a_1=\alpha (h, y_1^-)<\alpha (\varphi_1, z_1^-)=\alpha (g, z_1^-)$.

In the case $\tilde\alpha (g, y_1)>\tilde\alpha (h, y_1)$ take  $a_1<\tilde\alpha (h, y_1)$ such that $a_1\in X^0$,  $(a_1, \tilde\alpha (h, y_1))\in\Omega_{\theta}$  (and hence, $(a_1, \tilde\alpha (g, y_1))\in\Omega_{\theta}$). We assume that $y_1^-=\alpha (h^{-1}, a_1)$ and take $y_1^-<z_1^-<y_1$, $z_1^-\in X^0$, $a_1<\alpha (g, z_1^-)$ (using continuity of $g$ and the inequality $a_1<\tilde\alpha (g, y_1)$). Since $X$ is ultrahomogeneous, there exists $\varphi_1\in G$ such that $\alpha (\varphi_1, y_1^-)=a_1=\alpha (h, y_1^-)<\alpha (\varphi_1, z_1^-)=\alpha (g, z_1^-)$. Put 
$$\tilde\alpha (g', y)=\left\{\begin{array}{lcl}
\tilde\alpha (h, y) & \mbox{if} & y\in [\inf, y_1^-], \\
\tilde\alpha (\varphi_1, y) & \mbox{if} & y\in [y_1^-, z_1^-],\\
\tilde\alpha (g, y) & \mbox{if} & y\in [z_1^-, y_1].\\
\end{array}\right.$$

On the interval $[y_n, \sup]$ the construction is similar. 

\medskip

(II)  The map $g'$ on on the interval $[y_k, y_{k+1}]$, $k=1, \dots, n-1$. If 
$$(\tilde\alpha (h, y_k), \tilde\alpha (h, y_{k+1}))\in\Omega_{\theta},\ \mbox{or}$$  
$\tilde\alpha (h, y_k)\in\theta$, $\tilde\alpha (h, y_{k+1})\in\theta$ (and hence,  $\tilde\alpha (g, y_k)\in\theta$, $\tilde\alpha (g, y_{k+1})\in\theta$, $y_k=(x_k, 0)$, $y_{k+1}=(x_{k+1}, 0)$,  $\tilde\alpha (h, (x_k, 1))=\tilde\alpha (g, (x_k, 1))$, $\tilde\alpha (h, (x_{k+1}, -1))=\tilde\alpha (g, (x_{k+1}, -1))$) and 
$(\tilde\alpha (h, (x_k, 1)), \tilde\alpha (h, (x_{k+1}, -1)))\in\Omega_{\theta}$, or 

\noindent $\tilde\alpha (h, y_k)\in\theta$, $\tilde\alpha (h, y_{k+1})\not\in\theta$ and $(\tilde\alpha (h, (x_k, 1)), \tilde\alpha (h, y_{k+1}))\in\Omega_{\theta}$, or 

\noindent $\tilde\alpha (h, y_k)\not\in\theta$, $\tilde\alpha (h, y_{k+1})\in\theta$ and $(\tilde\alpha (h, y_k), \tilde\alpha (h, (x_{k+1}, -1)))\in\Omega_{\theta}$, then
$$\tilde\alpha (g', y)=\tilde\alpha (g, y)\ \mbox{for}\ y\in [y_k, y_{k+1}].$$  

Otherwise, let us examine the case $\tilde\alpha (h, y_k)\geq\tilde\alpha (g, y_k)$,  $\tilde\alpha (h, y_{k+1})\geq\tilde\alpha (g, y_{k+1})$. 

(a) If $\tilde\alpha (h, y_k), \tilde\alpha (h, y_{k+1})\not\in\theta$ (and, hence, $\tilde\alpha (g, y_k), \tilde\alpha (g, y_{k+1})\not\in\theta$), then take $b_k>\tilde\alpha (h, y_k)$, $b_k\in X^0$, $(b_k, \tilde\alpha (h, y_k))\in\Omega_{\theta}$ and $a_{k+1}<\tilde\alpha (g, y_{k+1})$, $a_{k+1}\in X^0$, $(a_{k+1}, \tilde\alpha (g, y_{k+1}))\in\Omega_{\theta}$. $b_k<a_{k+1}$.  Put $t=y_k$ and $\tau=y_{k+1}$. 

(b) If  $\tilde\alpha (h, y_k)$, $\tilde\alpha (h, y_{k+1})\in\theta$,  then take $b_k>\tilde\alpha (h, (x_k, 1))$, $b_k\in X^0$, $(b_k, \tilde\alpha (h, (x_k, 1)))\in\Omega_{\theta}$ and $a_{k+1}<\tilde\alpha (g, (x_{k+1}, -1))$, $a_{k+1}\in X^0$, $(a_{k+1}, \tilde\alpha (g, (x_{k+1}, -1)))\in\Omega_{\theta}$. $b_k<a_{k+1}$.  Put $t=(x_k, 1)$ and $\tau=(x_{k+1}, -1)$. 

(c) If  $\tilde\alpha (h, y_k)\in\theta$, $\tilde\alpha (h, y_{k+1})\not\in\theta$, then take $b_k>\tilde\alpha (h, (x_k, 1))$, $b_k\in X^0$, $(b_k, \tilde\alpha (h, (x_k, 1)))\in\Omega_{\theta}$ and $a_{k+1}<\tilde\alpha (g, y_{k+1})$, $a_{k+1}\in X^0$, $(a_{k+1}, \tilde\alpha (g, y_{k+1}))\in\Omega_{\theta}$. $b_k<a_{k+1}$.  Put $t=(x_k, 1)$ and $\tau=y_{k+1}$. 

(d) If  $\tilde\alpha (h, y_k)\not\in\theta$, $\tilde\alpha (h, y_{k+1})\in\theta$,  take $b_k>\tilde\alpha (h, y_k)$, $b_k\in X^0$, $(b_k, \tilde\alpha (h, y_k))\in\Omega_{\theta}$ and $a_{k+1}<\tilde\alpha (g, (x_{k+1}, -1))$, $a_{k+1}\in X^0$, $(a_{k+1}, \tilde\alpha (g, (x_{k+1}, -1)))\in\Omega_{\theta}$. $b_k<a_{k+1}$.  Put $t=y_k$ and $\tau=(x_{k+1}, -1)$. 

We assume that $y_k^+=\alpha (h^{-1}, b_k)$, $y_{k+1}^-=\alpha (h^{-1}, a_{k+1})$. From the continuity of $g$ and the inequality $\tilde\alpha (g, t)<b_k$ there exists $t<z^+_k<y_k^+$, $z^+_k\in X^0$, such that $\alpha (g, z_k^+)<b_k$, and from the inequality $a_{k+1}<\tilde\alpha (g, \tau)$ there exists $y_{k+1}^-<z^-_{k+1}<\tau$, $z^-_{k+1}\in X^0$, such that $a_{k+1}<\alpha (g, z_{k+1}^-)$. Since $X$ is ultrahomogeneous, there exist $\varphi_k^-\in G$ such that $\alpha (\varphi_k^-, z^+_k)=\alpha (g, z_k^+)$, $\alpha (\varphi_k^-, y_k^+)=b_k=\alpha (h, y_k^+)$ and 
$\varphi_k^+\in G$ such that $\alpha (\varphi_k^+, y_{k+1}^-)=a_{k+1}=\alpha (h, y_{k+1}^-)<\alpha (\varphi_k^+, z^-_{k+1})=\alpha (g, z_{k+1}^-)$. 
Put
$$g'(x)=\left\{\begin{array}{lcl}
\tilde\alpha (g, y) & \mbox{if} & y\in [y_k, z_k^+], \\
\tilde\alpha (\varphi_k^-, y) & \mbox{if} & y\in [z_k^+, y_k^+], \\
\tilde\alpha (h, y) & \mbox{if} & y\in [y_k^+, y_{k+1}^-], \\
\tilde\alpha (\varphi_k^+, y) & \mbox{if} & x\in [y_{k+1}^-, z_{k+1}^-], \\
\tilde\alpha (g, y) & \mbox{if} & y\in [z_{k+1}^-, y_{k+1}]. \\
\end{array}\right.$$ 

In the cases $\tilde\alpha (h, y_k)\leq\tilde\alpha (g, y_k)$,  $\tilde\alpha (h, y_{k+1})\leq\tilde\alpha (g, y_{k+1})$,  $\tilde\alpha (h, y_k)\geq\tilde\alpha (g, y_k)$,  $\tilde\alpha (h, y_{k+1})\leq\tilde\alpha (g, y_{k+1})$ and  $\tilde\alpha (h, y_k)\leq\tilde\alpha (g, y_k)$, $\tilde\alpha (h, y_{k+1})\geq\tilde\alpha (g, y_{k+1})$  considerations are similar. 

\medskip

It follows from the construction and the equality $\Omega_{\sigma}\circ\Omega_{\sigma}=\Omega_{\sigma}$ that $g'\in G$, $g'\in g\st_{\sigma}$ and $(\tilde\alpha (h, y), \tilde\alpha (g', y))\in\Omega_{\sigma}$ for any point $y\in b_m X$. The condition  $(\star\star)$ is valid for $(G, (b X, \tilde{\mathcal U}_X), \tilde\alpha)$. Hence, $\widetilde{\mathfrak{E}}_{\alpha}(\mathcal U_X)= R_{\Sigma_{b_m X}}$.

By Corollary~\ref{cora2-2} the condition  $(\star\star)$ is valid for $(G, (X, \mathcal U_X), \alpha)$ and   $\mathfrak{E}_{\alpha}(\mathcal U_X)=R_{\Sigma_X}$. Hence, all equalities and inclusions in (b) and (c) are proved.

\medskip

From item (4) of Theorem~\ref{autultr} it follows that  $\widetilde{{\mathfrak E}}_{\star}(L\wedge R)=\widetilde{{\mathfrak E}}_{\alpha}(\mathcal U_X)$. 

If $R_{\Sigma_X}\ne R_{\Sigma_{b_m X}}$, then  $\mathfrak{E}_{\alpha}(\mathcal U_X)\ne\widetilde{\mathfrak{E}}_{\alpha}(\mathcal U_X)$, $\mathfrak{E}_{\alpha}(\mathcal U_X)=L\wedge R\ne \widetilde{{\mathfrak E}}_{\star}(L\wedge R)=\widetilde{{\mathfrak E}}_{\alpha}(\mathcal U_X)$. Hence,  by item (e) of Theorem~\ref{mathfrak} $L\wedge R\not\in\mathbb{E}(G)$ and (c) is proved.

\medskip 

(d) Fulfillment of conditions (i)--(v) for $f\in\cl~G$. 

(i) Let $x<y$ and $f(x)>f(y)$. Take disjoint nbds $Of(x)$, $Of(y)$ of $f(x)$ and $f(y)$ respectively, which are intervals. Then $O_f=[x,  Of(x)]\cap [y,  Of(y)]$ is a nbd of $f$, but $G\cap O_f=\emptyset$. Hence, $f\not\in\cl~G$. 

(ii) Let $x\in b_m X\setminus X$ and $f(x)\in X$. Then  $O_f=[x, \{f(x)\}]$ is a nbd of $f$,  but $G\cap O_f=\emptyset$. Hence, $f\not\in\cl~G$. 

(iii) Let $f(x)=f(y)$, $x\ne y$, and $f(x)\in X$. Then  $O_f=[x,  \{f(x)\}]\cap [y,  \{f(y)\}]$ is a nbd of $f$, but $G\cap O_f=\emptyset$.  Hence, $f\not\in\cl~G$.

(iv) If $f((x, -1))=f((x, 0))<f((x, 1))$, then $f((x, 0))$, $f((x, 1))\in b_m X\setminus X$ by (ii) and  then there exists $f((x, 0))< z<f((x, 1))$. For the disjoint nbds $U=[\inf, z)$ and $V=(z, \sup]$ of $f((x, 0))$ and $f((x, 1))$ respectively,  $O_f=[(x, 0),  U]\cap [(x, 1),  V]$ is a nbd of $f$, but $G\cap O_f=\emptyset$. 

This argument works in the case $f((x, -1))<f((x, 0))=f((x, 1))$. Therefore, either $f((x, -1))=f((x, 0))=f((x, 1))\not\in X$ or $f((x, -1))<f((x, 0))<f((x, 1))$ and  there is no $z$ (by the above argument) such that $f((x, -1))<z<f((x, 0))$ or $f((x, 0)<z<f((x, 1))$.  Hence,  $f((x, -1))=(y, -1)$, $f((x, 0))=(y, 0)$,  $f((x, 1))=(y, 1)$. 

(v) Let $f(\inf)=x\ne\inf$. Take nbd $Ox$ of $x$, $\inf\not\in Ox$. Then $O_f=[\inf, Ox]$ is a nbd of $f$,  but  $G\cap O_f=\emptyset$. Hence, $f\not\in\cl~G$. The same argument works for $\sup$. 

\medskip

Let $f: b_m X\to b_m X$ satisfies conditions (i)--(v).  Take an arbitrary nbd $O_f=\bigcap\limits_{k=1}^n [x_k, Of(x_k)]$ of $f$ 
where, without lose of generality, $\inf<x_1<\ldots<x_n<\sup$, for any $x\in X$ only one point of the form $(x, j)$, $j=-1, 0, 1$, is in the set $\{x_1, \ldots, x_n\}$ (use (iv)), $Of(x_k)$, $k=1,\ldots, n$ are clopen intervals, $Of(x_k)=Of(x_{k+1})$ if $f(x_k)=f(x_{k+1})$ and  $Of(x_k)\cap Of(x_{k+1})=\emptyset$ if $f(x_k)<f(x_{k+1})$,  $k=1,\ldots, n-1$ (use condition (i)). Moreover, if $f(x_k)=(y_k, 0)$ (and hence, $x_k\in X^0$), then $Of(x_k)=\{(y_k, 0)\}$, if $f(x_k)=(y_k, 1)$, then $Of(x_k)=[(y_k, 1), (z_k, -1)]$,  if $f(x_k)=(y_k, -1)$, then $Of(x_k)=[(z_k, 1), (y_k, -1)]$.

In order to construct automorphism $g\in O_f$ we sequentially determine two finite increasing sequences of points $t_k\leq t'_k<t_{k+1}$ and $\tau_k\leq\tau'_k<\tau_{k+1}$ in $X$, $k=1,\ldots, n$. 

If $Of(x_k)=\{(y_k, 0)\}$, then $x_k=(t_k, 0)$ by (ii), $t'_k=t_k$, and $\tau_k=\tau'_k=y_k$.

Let $Of(x_k)=[(y_k, 1), (z_k, -1)]$. If 
\begin{itemize}
\item[{\rm ($\alpha$)}] $x_k=(r_k, -1)$ and $f(x_k)\ne (z_k, -1)$, or $f((r_k, 0))=(z_k, -1)$ ($f((r_k, -1))=f((r_k, 1))=(z_k, -1)$ by (iv)),  
\item[{\rm ($\beta$)}] $x_k=(r_k, 0)$, 
\item[{\rm ($\gamma$)}] $x_k=(r_k, 1)$ and $f(x_k)\ne (y_k, 1)$, or $f((r_k, 0))=(y_k, 1)$ ($f((r_k, -1))=f((r_k, 1))=(z_k, 1)$ by (iv)),  
\end{itemize}
then take any $(\tau_k, 0)\in [(y_k, 1), (z_k, -1)]$, $\tau'_{k-1}<\tau_k$ (if $k-1=0$ , then $\inf<(\tau_k, 0)$) and put $t_k=t'_k=r_k$, $\tau_k=\tau'_k$. 

If $x_k=(r_k, -1)$, $f(x_k)= (z_k, -1)$ and $f((r_k, 0))=(z_k, 0)$ (by (iv)), then put $t_k=t'_k=r_k$, $\tau_k=\tau'_k=z_k$. 

If $x_k=(r_k, 1)$, $f(x_k)=(y_k, 1)$ and $f((r_k, 0))=(y_k, 0)$  (by (iv)),  then put $t_k=t'_k=r_k$, $\tau_k=\tau'_k=y_k$. 

If  $x_k\in\Gamma$, then take any $(\tau_k, 0),\ (\tau'_k, 0)\in [(y_k, 1), (z_k, -1)]$,  $(\tau'_{k-1}, 0)<(\tau_k, 0)$ and any $t_k,\ t'_k\in X$ such that $(t'_{k-1}, 0)<(t_k, 0)<x_k<(t'_{k}, 0)<x_{k+1}$ (if $k-1=0$ , then $\inf<(t_k, 0)$; if $k=n$, then $(t'_{k}, 0)<\sup$).  

Since $X$ is ultrahomogeneous there exists  $g\in G$ that maps ordered sequence of $t_k,\ t'_k$ onto the ordered sequence of  $\tau_k,\ \tau'_k$. Evidently, $g\in O_f$ and $f\in\cl~G$.

\medskip

(e) Since $L\wedge R=R_{\Sigma_X}$, the Roelcke compactification $b_r G$  of $G$ is the closure of $\imath (G)$ in $(b_m X)^{X}$. The restriction of the  projection $\pr: (b_m X)^{b_m X}\to (b_m X)^{X}$ to $\cl~\tilde{\imath} (G)$ is a map of the proper Ellis semigroup compactification $b G$ onto $b_r G$. This gives the description of $b_r G$ as the set of maps $f$ of $X$ to $b_m X$  in the t.p.c.  such that  (a) $f$ is monotone,  (b) if $f(x)=f(y)$, then $f(x)\in b_m X\setminus X$.

\medskip

If $X$ is continuous, then (b') follows from the equalities 
$\widetilde{\mathfrak{E}}_{\alpha}(\mathcal U_X)=\mathfrak{E}_{\alpha}(\mathcal U_X)$ and $R_{\Sigma_X}=R_{\Sigma_{b_m X}}$. 

$R_{\Sigma_X}=R_{\Sigma_{b_m X}}$, because $\st_{\inf}={\rm St}_{\sup}=G$ and $\st_{(x, -1)}=\st_{(x, 1)}=\st_{(x, 0)}$ for $x\in X$.

Since $\tilde\alpha (g, \inf)=\inf$, $\tilde\alpha (g, \sup)=\sup$, $g\in G$, in order to prove $\widetilde{\mathfrak{E}}_{\alpha}(\mathcal U_X)=\mathfrak{E}_{\alpha}(\mathcal U_X)$ it is enough to check the fulfillment of the condition $(\star)$ of Proposition~\ref{order}. Let us show that for any $\Omega\in\tilde{\mathcal U}_X$ and any point $x\in X$ the following hold
$$(\tilde\alpha (g, (x, -1)), \tilde\alpha (f, (x, -1))\in\Omega\ \mbox{iff}\ (\tilde\alpha (g, (x, 1)), \tilde\alpha (f, (x, 1))\in\Omega\ \mbox{iff}$$
$$(\tilde\alpha (g, (x, 0)), \tilde\alpha (f, (x, 0))\in\Omega.$$
Let, for example, $(x, 1)\in b_m X$ and an arbitrary cover $\omega$ of $\{(x, 1)\}\times b_m X$ from $\tilde{\mathcal U}_X$  be of the form 
$$\omega=[\inf, (x_1, -1)]\bigcup\{(x_1, 0)\}\bigcup [(x_1, 1), (x_2, -1)]\bigcup\{(x_2, 0)\}\bigcup\ldots$$
$$\ldots\bigcup[(x_{n-1}, 1), (x_n, -1)]\bigcup\{(x_n, 0)\}\bigcup [(x_n, 1), \sup],$$
where $x_1<x_2<\ldots<x_{n-1}<x_n$. It consists of disjoint sets. 
Let us note that if $(y, 0)\in [\inf, (x_1, -1)]$ or $[(x_k, 1), (x_{k+1}, -1)]$, $k=1,\ldots, n-1$, or $[(x_n, 1), \sup]$, then  $(y, -1), (y, 1)\in [\inf, (x_1, -1)]$ or $ [(x_k, 1), (x_{k+1}, -1)]$, $k=1,\ldots, n-1$, or $[(x_n, 1), \sup]$ respectively. 

Take  $(x, 0)\in b_m X$ and the same cover $\omega$ on  $\{(x, 0)\}\times b_m X$. 

If $f(x, 0)=g(x, 0)$, then  $f(x, 1)=g(x, 1)$.  In particular, if  $f(x, 0)=g(x, 0)=(x_k, 0)$, then $f(x, 1)=g(x, 1)=(x_k, 1)\in [(x_k, 1), (x_{k+1}, -1)]$, $k=1,\ldots, n$.

If $f(x, 0)=(y, 0), g(x, 0)=(z, 0)\in [\inf, (x_1, -1)]$ or $[(x_k, 1), (x_{k+1}, -1)]$, $k=1,\ldots, n-1$, or $[(x_n, 0), \sup]$, then $f(x, 1)=(y, 1)$, $g(x, 1)=(z, 1)$. The pair of points $(y, 0), (y, 1)$ (and $(z, 0), (z, 1)$) is in one element of the cover $\omega$, the pair of points  $(y, 0), (z, 0)$ is in one element of the cover $\omega$. Hence, the pair of points $f(x, 1)=(y, 1)$, $g(x, 1)=(z, 1)$ is in one element of the cover $\omega$. 

In the case $(x, -1)\in b_m X$ the same reasoning can be used and (b') is proved.

\medskip

(c') follows from (b'). 
\end{proof}

\begin{rem}{\rm $b G$ is not a semitopological semigroup. 

Indeed, take the map  $f\in b G$,  $f(\inf)=\inf$, $f(x)=\sup$ otherwise. Any nbd of the map $g\in b G$, $g(\sup)=\sup$, $g (x)=\inf$ otherwise, is of the form  $O_x^y=\{h\in b G\ |\ h((x, 0))<(y, 0)\}$, $x, y\in X$. Since $f\circ g=g$ in any nbd of $g$ of the form of $O_x^y$ there is $h\in b G$ such that $h(\inf)=\inf$, $h(x)=(y, -1)$, $\inf<x<\sup$,  $h(\sup)=\sup$. Hence, $f\circ h(\inf)=\inf$, $f\circ h(x)=\sup$ and can't belong to the arbitrary nbd of $g$ the form $O_x^y$. Thus, multiplication on the left in $b G$ is not continuous.}
\end{rem}

\begin{cor}\label{contchain}
 Let $X$ be an  ultrahomogeneous chain and a discrete space, $G=(\aut(X), \tau_{\partial})$. Then 
\begin{itemize}
\item[{\rm(1)}] there is no $\mathcal U\in \mathbb{E}(G)$ such that $L\wedge R\subsetneq\mathcal U\subsetneq\widetilde{{\mathfrak E}}_{\alpha}(\mathcal U_X)$,
\item[{\rm(2)}]  $L\wedge R\in\mathbb{E}(G)$ iff $X$ is continuous.
\end{itemize}
\end{cor}

\begin{proof} (1) follows from the equality $\widetilde{{\mathfrak E}}_{\star}(L\wedge R)=\widetilde{{\mathfrak E}}_{\alpha}(\mathcal U_X)$ (item (c) of Theorem~\ref{Roelcke precomp4-2-1}) and items (d), (f) of Theorem~\ref{mathfrak}.

Sufficiency in (2) is proved in item (b') of Theorem~\ref{Roelcke precomp4-2-1}. Necessity. Suppose that $X$ is not continuous. It is enough to show that  $R_{\Sigma_X}\ne R_{\Sigma_{b_m X}}$ by item (c) of  Theorem~\ref{Roelcke precomp4-2-1}.

Take $x\in X$ and let $\omega=[\inf, (x, -1)]\cup\{(x, 0)\}\cup[(x, 1), \sup]\in\tilde{\mathcal U}_X$, $\xi$ is a proper gap in $X$, $O\in N_G(e)$ is such that $\{Ox\ |\ x\in b_m X\}\succ\omega$.

For any $\sigma\in\Sigma_X$ let $(x_1, 0)<\xi<(x_2, 0)$, $x_1, x_2\in\sigma$ and $\sigma\cap ((x_1, 0), (x_2, 0))=\emptyset$. Take  $y_1, y_2, y_1', y_2'\in X$ such that $(x_1, 0)<(y_1, 0)<(y_1', 0)<\xi<(y_2, 0)<(y_2', 0)<(x_2, 0)$ and $g\in G$ such that $g((y_2, 0))<(x, 0)$, $g((y_2', 0))>(x, 0)$. Let $h\in G$ sends  points  $x_1<y_2<y_2'<x_2$ to  $x_1<y_1<y_1'<x_2$ and identity on the points from $\sigma$. Then $h\in\st_{\sigma}$ and, hence, for any $V\in N_G(e)$ $gh\in Vg\st_{\sigma}$. Further, $g(\xi)<(x, 0)$, $gh(\xi)>(x, 0)$. Therefore, $gh\not\in Og\st_{\xi}$. 

It is shown that for any $\sigma\in\Sigma_X$  and  $V\in N_G(e)$ the cover $\{Vg\st_{\sigma}\ |\ g\in G\}$ is not refined in the cover $\{Og\st_{\xi}\ |\ g\in G\}$. Hence, $L\wedge R=R_{\Sigma_X}\ne R_{\Sigma_{b_m X}}=\widetilde{{\mathfrak E}}_{\star}(L\wedge R)$.
\end{proof}

\begin{rem}\label{descrchain}  {\rm  Corollary~\ref{contchain} shows that an oligomorphic action (see~\cite{Tsan} and~\cite{Sorin1}) doesn't guarantee that the Roelcke compactification of an acting group is a proper Ellis semigroup compactification.}
\end{rem}

\begin{df} An action of a group $G\subset\aut X$ on a chain $X$ is strongly $n$-transitive, $n\geq 1$, if for any families of $n$ points $x_1< \ldots<x_n$ and $y_1<\ldots<y_n$ there exists $g\in G$ such that $g(x_k)=y_k$, $k=1,\ldots, n$. 

An action $G\curvearrowright X$, which is strongly $n$-transitive for all $n\in\mathbb N$, is called ultratransitive.
\end{df}

\begin{rem}\label{denseultr}
{\rm If an action of $G$ on a chain $X$ is  ultratransitive, then  $(G, \tau_{\partial})$ is a dense subgroup $(\aut (X), \tau_{\partial})$ and $X$ is a coset space of  $(G, \tau_{\partial})$.

Indeed, for any $g\in\aut (X)$ and a nbd $\{h\in\aut (X)\ |\ h(x_i)=g(x_i),\ i=1,\ldots, n\}$ of $g$, ultratransitivity of the action $G\curvearrowright X$ yields that there exists $f\in G\cap\{h\in\aut (X)\ |\ h(x_i)=g(x_i),\ i=1,\ldots, n\}$.}
\end{rem}

\begin{cor}\label{ccc}
Let  an action $\alpha_G: (G, \tau_{\partial})\times X\to X$ on a chain $X$ {\rm(}discrete space{\rm)} be ultratransitive, $\mathcal U_X^G$ is the maximal equiuniformity on $X$. Then
\begin{itemize}
\item[{\rm (a)}]  $\mathcal U_X=\mathcal U_X^G$, where $\mathcal U_X$ is the maximal equiuniformity on $X$  for the action $\alpha: (\aut (X), \tau_{\partial})\times X\to X$, 
\item[{\rm (b)}] $\mathfrak{E}_{\alpha}(\mathcal U_X)|_{G}=\mathfrak{E}_{\alpha_G}(\mathcal U_X^G)=(L\wedge R)_G$ and  $\widetilde{\mathfrak{E}}_{\alpha}(\mathcal U_X)|_{G}=\widetilde{\mathfrak{E}}_{\alpha_G}(\mathcal U_X^G)$, where $(L\wedge R)_G$ is the Roelcke uniformity on $G$, 
\item[{\rm (c)}] $\widetilde{\mathfrak{E}}_{\alpha_G}(\mathcal U_X^G)=\widetilde{{\mathfrak E}}_{\star}((L\wedge R)_G)$, 
\item[{\rm (d)}]  $b G$ and $b\, \aut (X)$ are topologically isomorphic proper enveloping Ellis semigroup compactifications.  
\end{itemize}
If $X$ is continuous, then the Roelcke compactifications $b_r G$ and $b_r\, \aut(X)$ are topologically isomorphic proper enveloping Ellis semigroup compactifications. If the Roelcke compactifications $b_r G$ is a proper enveloping Ellis semigroup compactification, then $X$ is continuous. 
\end{cor}

\begin{proof} (a) follows from ultratransitivity of action. 

(b) $\mathfrak{E}_{\alpha}(\mathcal U_X)|_{G}=\mathfrak{E}_{\alpha_G}(\mathcal U_X^G)$ and  $\widetilde{\mathfrak{E}}_{\alpha}(\mathcal U_X)|_{G}=\widetilde{\mathfrak{E}}_{\alpha_G}(\mathcal U_X^G)$ in (b) follows from (a) and the definition of Ellis equiuniformity. $L\wedge R|_G=(L\wedge R)_G$ by Remark~\ref{denseultr} and~\cite[Proposition 3.24]{RD} and Roelcke precompact. Since $\mathfrak{E}_{\alpha}(\mathcal U_X)=L\wedge R$ (item (b) of Theorem~\ref{Roelcke precomp4-2-1}), $\mathfrak{E}_{\alpha}(\mathcal U_X)|_{G}=(L\wedge R)_G$ and (b) is proved. 

(c) Since $(L\wedge R)_G\subset\widetilde{\mathfrak{E}}_{\alpha_G}(\mathcal U_X^G)\subset\widetilde{{\mathfrak E}}_{\star} ((L\wedge R)_G)$ by (b) and items (1), (2) of Theorem~\ref{autultr},  $\widetilde{\mathfrak{E}}_{\alpha_G}(\mathcal U_X^G)=\widetilde{{\mathfrak E}}_{\star}((L\wedge R)_G)$  (item (f) of Theorem~\ref{mathfrak}).

(d) follows from (b) and Remark~\ref{denseultr}.

\medskip

$b_r G$ and $b_r\, \aut(X)$ are uniformly isomorphic by the extension of the identity map on $G$ (since $G$ is dense in $\aut\, (X)$). 

If $X$ is continuous $\tilde{\mathfrak{E}}_{\alpha}(\mathcal U_X)=L\wedge R$ and $b_r\, \aut(X)$  is a proper enveloping Ellis semigroup compactification (items (b') and (c') of Theorem~\ref{Roelcke precomp4-2-1}). 

If  $b_r G$ is a proper enveloping Ellis semigroup compactification of $G$, then $b_r\, \aut(X)$ is a proper enveloping Ellis semigroup compactification of $\aut\, (X)$ ($b_r\, \aut(X)$ is a right topological semigroup and the action $G\curvearrowright b_r\, \aut(X)$ is extended to the action  $\aut(X)\curvearrowright b_r\, \aut(X)$, because $G$ is dense in $\aut(X)$ and  $\aut(X)$ is contained in the Raikov completion of $G$ to which the action is extended~\cite{Megr0}). By Corollary~\ref{contchain} $X$ is continuous and the last statement is proved. 
\end{proof}

\begin{ex}\label{Thompson}
{\rm The action of the Thompson group $(F, \tau_{\partial})$ is ultratransitive on the chain (not continuous)  $\mathbb Q_r$ of dyadic rational numbers of $(0, 1)$ and every automorphism from $F$ is a piecewise-linear map that have a finite number of breakpoints~\cite{Canon}.

By Theorem~\ref{Roelcke precomp4-2-1} and Corollary~\ref{ccc} $F$ is Roelcke precompact, however, $L\wedge R\not\in\mathbb{E}(F)$ and the least proper enveloping Ellis semigroup compactification which is greater than the Roelcke compactification of $F$ is the completion of $(F, \widetilde{\mathfrak{E}}_{\alpha_F}(\mathcal U_{\mathbb Q_r}^F))$. 

The description of $b_m\mathbb Q_r$ is given above. Clopen intervals with endpoints whose first coordinates (following the description) are dyadic rational numbers form the countable base of the topology of $b_m\mathbb Q_r$. Hence,  $b_m\mathbb Q_r$ is a metrizable linearly ordered compacta (see, \cite[Example 5.4]{MegrS}).

$b_r F=b_r \aut\, (\mathbb Q_r)$ by item (b) of Corollary~\ref{ccc} and its description is given in item (e)  of Theorem~\ref{Roelcke precomp4-2-1}.}
\end{ex}


\subsection{The maps $\mathfrak{E}_\alpha$ and  $\tilde{\mathfrak{E}}_\alpha$ for the action of the unitary group on a Hilbert space}\label{UNIT}

Let $U({\bf H})$ be the unitary group of a (complex or real) Hilbert space  ${\bf H}$. It is a group of isometries for the action of $U({\bf H})$ on the unit sphere ${\rm S}\subset{\bf H}$, ${\rm S}$ is a coset space of  $(U({\bf H}), \tau_p)$ under its action~\cite[\S\ 2.1]{Kad} and the stabilizer $\st_x$ of any point $x\in {\rm S}$ is a neutral subgroup of $U({\bf H})$~\cite[Remark 3.14]{Kozlov}, ${\rm B}$ in  the unit ball in ${\bf H}$. The t.p.c. $ \tau_p$ coincides with the {\it strong operator topology} (see for example, \cite{Rud}) on  $U({\bf H})$ (since for every $x\ne 0$ 
$||g(x)-x||<\varepsilon\Longleftrightarrow ||g\big(\frac{x}{||x||}\big)-\frac{x}{||x||}||<\frac{\varepsilon}{||x||}$).

Let $\mathcal L_{\bf H}$ be the set of all finite-dimensional subspaces of ${\bf H}$, ${\rm B}_L=\{x\in L\ |\ ||x||\leq 1\}$ the {\it unit ball},  ${\rm S}_L=\{x\in L\ |\ ||x||=1\}$ is the {\it unit sphere} in $L$, $L\in\mathcal L_{\bf H}$.  

\begin{lem}\label{baseunit}  The sets $V_{L, \varepsilon}=\{ g\in U({\bf H})\ |\ ||g(x)-x||<\varepsilon,\ x\in {\rm B}_L\}$, $L\in\mathcal L_{{\bf H}}$, $\varepsilon>0$, form a base at the unit in  $(U({\bf H}), \tau_p)$.
\end{lem}

\begin{proof} If $V=\{ g\in U({\bf H})\ |\ ||g(x)-x||<\varepsilon,\ x\in\sigma\}\in N_{U({\bf H})}(e)$, where $\sigma=\{x_1,\ldots, x_n\}\in\Sigma_{S}$ and  $L={\rm span}(x_1, \dots, x_n)$ (the smallest linear subspace that contains $\sigma$), then $V_{L, \varepsilon}\subset V$.

For $V_{L, \varepsilon}$ let $e_1,\ldots, e_n$ be the orthonormal basis of $L$. Then  $V=\{ g\in U({\bf H})\ |\ ||g(x)-x||<\varepsilon,\ x\in\sigma\}\subset V_{L, \varepsilon}$, where $\sigma=\{e_1,\ldots, e_n\}$. From this lemma follows. 
\end{proof}

\begin{pro}\label{unitmax}
The maximal equiuniformity $\mathcal U_{\rm S}$ on $S$ is totally bounded.
\end{pro}

\begin{proof}
Let $O\in N_{U({\bf H})}(e)$,  $V\in N_{U({\bf H})}(e)$ is such that $V^2\subset O$, and, without loss of generality, is of the form $V=V_{L, \varepsilon}$,  $L\in\mathcal L_{{\bf H}}$, $\varepsilon>0$, by Lemma~\ref{baseunit}, $\{Ox\ |\ x\in {\rm S}\}\in\mathcal U_{\rm S}$.

Let ${\bf H}={\rm L}\oplus {\rm L}^\perp$,  $\pr: {\bf H}\to {\rm L}$ and $\pr_\perp: {\bf H}\to {\rm L}^\perp$ be the orthogonal projections, $\dim L=n$. One can define a continuous map $F: {\bf H}\to {\rm L}\times\mathbb R_+$, $F(x)=(\pr(x), ||{\rm pr}_\perp (x)||)$. Since  $||\pr(x)||^2+||\pr_\perp (x)||^2=1$ for $x\in {\rm S}$, the restriction $f=F|_{\rm S}$ of $F$ to the unit sphere sends points of ${\rm S}$ to the compact set ${\rm S}_+^{n}=\{(t, \tau)\in L\times\mathbb R_+: ||t||^2+\tau^2=1\}$.

The map $f: {\rm S}\to {\rm S}_+^{n}$ is a composition of the projection $\pr|_{\rm S}: {\rm S}\to {\rm B}_L$ and the map $h: {\rm B}_L\to  {\rm S}_+^{n}$, $h(x)=(x, \sqrt{1-||x||^2})$, which is a homeomorphism of ${\rm B}_L$ and ${\rm S}_+^{n}$. To show that the map $f$ is open it is sufficient to show that  $\pr|_{\rm S}$ is open. 

Let us note that the map $\sqrt{1-t^2}: [0, 1]\to [0, 1]$ is uniformly continuous. Hence, for any $b>0$ there exists $a>0$ such that if $|t-t'|<a$, then $|\sqrt{1-t^2}-\sqrt{1-{t'}^2}|<b$.

Take $x\in {\rm S}$ and its nbd $W(x)={\rm S}\cap W$, where  $W=\{y\ |\ ||\pr(x)-{\rm pr}(y)||<a,  ||\pr_\perp(x)-\pr_\perp(y)||<b\}$. For any $x'\in {\rm B}_L$ such that $||\pr(x)-x'||<a$ take $y'=\frac{\sqrt{1-||x'||^2}}{||\pr_\perp(x)||}\pr_\perp(x)$. The point $(x', y')\in {\rm S}$ and  
$$||y'-\pr_\perp(x)||=|\sqrt{1-||x'||^2}-||\pr_\perp(x)|||=|\sqrt{1-||x'||^2}-\sqrt{1-||\pr(x)||^2}|<b,$$ 
since  $|||\pr(x)||-||x'|||\leq ||\pr(x)-x'||<a$ and the choice of $a$ and $b$. Thus, $\pr(W)\cap {\rm B}_L=\pr|_{\rm S}(W(x))$. This shows that  $\pr|_{\rm S}$ (and, hence, $f$) is open.

For the action  $\st_{L}\curvearrowright {\bf H}(={\rm L}\oplus {\rm L}^\perp)$ the group $\st_{L}$ (stabilizer of points of the subspace $L$) is naturally identified with the unitary group $U({\rm L}^\perp)$. The sphere ${\rm S}$ is an invariant subset which orbit space ${\rm S}/\st_{L}$ is identified with ${\rm S}_+^{n}$. Since the map of ${\rm S}$ onto the orbit space ${\rm S}/\st_{L}$ is open, ${\rm S}/\st_{L}$ is homeomorphic to ${\rm S}_+^{n}$. 

The cover $\{Vx\ |\ x\in {\rm S}\}\in\mathcal U_{\rm S}$, its image $\{f(Vx)\ |\ x\in {\rm S}\}$ is an open cover of the compactum ${\rm S}/\st_{L}$ and has a finite refinement $f(V_k)$, $k=1,\ldots, m$. Since $\st_L\subset V$, the finite open cover $f^{-1}(f(V_k))=\st_L V_k$, $k=1,\ldots, m$, of ${\rm S}$ is refined in  $\{Ox\ |\ x\in {\rm S}\}$. Therefore, $\mathcal U_{\rm S}$ is totally bounded.
\end{proof}

\begin{rem}{\rm Another reasoning (which uses Roelcke precompactness of $U({\bf H})$~\cite{U1998}) to prove that the maximal equiuniformity $\mathcal U_{\rm S}$ on $S$ is totally bounded can be found in~\cite[Example 3.31]{Kozlov}.}
\end{rem}

\begin{lem}\label{unifunit}
The base of the uniformity $\widetilde{\mathfrak{E}}_{\alpha}(\mathcal U_{\rm S})$ is formed by the entourages 
$$\{(f, g)\in U({\bf H})^2\ |\ |(f(x), y)-(g(x), y)|<\epsilon,\ x, y\in {\rm B}_L\},\ L\in \mathcal L_{\bf H},$$
where $(\cdot, \cdot)$ is the scalar product on  ${\bf H}$.
\end{lem}

\begin{proof} In~\cite[Corollary 2.2]{Stojanov} it is shown that the unit ball ${\rm B}$ in  ${\bf H}$ in the {\it weak topology} (the topology initial with respect to the projections onto finite-dimensional subspaces) is the maximal $G$-compactification of ${\rm S}$ (completion of $({\rm S}, \mathcal U_{\rm S})$). 

Therefore, the entourages from $\widetilde{\mathcal U}_{\rm S}$ are of the form 
$$\{(x, y)\in {\rm B}^2\ |\ ||\pr_L x-\pr_L y||<\epsilon\},\ L\in \mathcal L_{\bf H},$$
and the entourages from $\widetilde{\mathfrak{E}}_{\alpha}(\mathcal U_{\rm S})$ are of the form 
$$\{(f, g)\in U({\bf H})^2\ |\  ||\pr_L f(x)-\pr_L g(x)||<\epsilon,\ x\in\sigma\},\ \sigma\in\Sigma_{\rm B},\ L\in \mathcal L_{\bf H},$$
where $\pr_L$ is the orthogonal projection of ${\bf H}$ onto $L$. Without loss of generality, one can assume that $\sigma=\{x_1,\ldots, x_n\}$, $0\not\in\sigma$ and ${\rm span}\ (x_1,\ldots, x_n)=L$.
$$||\pr_L f(x)-\pr_L g(x)||=||\pr_L (f(x)-g(x))||=\sup\limits_{ z\in {\rm S}_L}|(f(x)-g(x), z)|=\sup\limits_{ z\in {\rm S}_L}|(f(x), z)-(g(x), z)|.$$
For $y\in {\rm B}_L$, $y\ne 0$
$$|(f(x), y)-(g(x), y)|=||y||\cdot|(f(x), \frac{y}{||y||})-(g(x), \frac{y}{||y||})|\leq |(f(x), \frac{y}{||y||})-(g(x), \frac{y}{||y||})|.$$
Therefore, $||\pr_L f(x)-\pr_L g(x)||=\sup\limits_{ y\in {\rm B}_L}|(f(x), y)-(g(x), y)|$ and 
$$\{(f, g)\in U({\bf H})^2\ |\ |(f(x), y)-(g(x), y)|<\epsilon,\ x, y\in {\rm B}_L\}\subset$$
$$\subset\{(f, g)\in U({\bf H})^2\ |\  ||\pr_L f(x)-\pr_L g(x)||<\epsilon,\ x\in\sigma\}.$$

For $\{(f, g)\in U({\bf H})^2\ |\ |(f(x), y)-(g(x), y)|<\epsilon,\ x, y\in {\rm B}_L\},\ L\in \mathcal L_{\bf H}$, let $e_1,\ldots, e_n$ be the orthonormal basis of $L$,  $\sigma=\{e_1,\ldots, e_n\}$. 
For any $x\in {\rm B}_L$, one has $x=\lambda_1 e_1+\ldots+\lambda_n e_n$, where $|\lambda_k|\leq 1$, $k=1,\ldots, n$, and 
$$|(f(x), y)-(g(x), y)|\leq\sum\limits_{k=1}^n |\lambda_k| |(f(e_k), y)-(g(e_k), y)|\leq\sum\limits_{k=1}^n|(f(e_k), y)-(g(e_k), y)|.$$
Therefore, if $||\pr_L f(x)-\pr_L g(x)||<\frac{\epsilon}n$, $x\in\sigma$, then $|(f(x), y)-(g(x), y)|\leq\varepsilon$ and 
$$\{(f, g)\in U({\bf H})^2\ |\  ||\pr_L f(x)-\pr_L g(x)||<\frac{\epsilon}n,\ x\in\sigma\}\subset$$
$$\{(f, g)\in U({\bf H})^2\ |\ |(f(x), y)-(g(x), y)|<\epsilon,\ x, y\in {\rm B}_L\}.$$
\end{proof}

\begin{thm}\label{Roelcke precompHilb}  For the action $\alpha:  U({\bf H})\times {\rm S}\to {\rm S}$ of the unitary group $U({\bf H})$ of a Hilbert space  ${\bf H}$ on the the unit sphere ${\rm S}$ and the maps $\mathfrak{E}_{\alpha}: \mathbb{BU}({\rm S})\to\mathbb{BU}(U({\bf H}))$,  $\widetilde{\mathfrak{E}}_{\alpha}: \mathbb{BU}({\rm S})\to\mathbb E (U({\bf H}))$ one has:
\begin{itemize}
\item[{\rm (a)}]  $\widetilde{\mathfrak{E}}_{\alpha}(\mathcal U_{\rm S})=\mathfrak{E}_{\alpha}(\mathcal U_{\rm S})=L\wedge R=R_{\Sigma_{\rm S}}=R_{\Sigma_{\rm B}}$ {\rm(}and, hence, $(U({\bf H}), \tau_p)$ is Roelcke precompact~{\rm\cite{U1998}}, and $b_r U({\bf H})$ is a proper enveloping Ellis semigroup compactification{\rm)}, 
\item[{\rm (b)}] $b_r U({\bf H})$ is the space of all linear operators on  ${\bf H}$  of norm $\leq 1$ in the weak operator topology~{\rm\cite{U1998}},
\item[{\rm (c)}] $b_r U({\bf H})$ is a semitopological semigroup~{\rm\cite{U1998}}.
\end{itemize}
\end{thm}

\begin{proof} 
By Proposition~\ref{unitmax} and item (1) of Theorem~\ref{autultr} to prove (a) one can check the fulfillment of the inclusions $\widetilde{\mathfrak{E}}_{\alpha}(\mathcal U_{\rm S})\subset\mathfrak{E}_{\alpha}(\mathcal U_{\rm S})$, $\mathfrak{E}_{\alpha}(\mathcal U_{\rm S})\supset R_{\Sigma_{\rm S}}$ and $R_{\Sigma_{\rm S}}\supset R_{\Sigma_{\rm B}}$.

The proof of the inclusions $\widetilde{\mathfrak{E}}_{\alpha}(\mathcal U_{\rm S})\subset\mathfrak{E}_{\alpha}(\mathcal U_{\rm S})$ and $R_{\Sigma_{\rm S}}\supset R_{\Sigma_{\rm B}}$.  For any $y\in{\rm B}$, $y\ne 0$, let $x=\frac{y}{||y||}\in{\rm S}$. 

Then $\st_y=\st_x$ ($g(y)=y\Longleftrightarrow g(x)=x$, $g\in U({\bf H})$). Since  $\st_0=U({\bf H})$,   $R_{\Sigma_{\rm S}}=R_{\Sigma_{\rm B}}$. 

If $f(x)\in Og(x)$, $O\in N_{U({\bf H})}(e)$, then there exists $h\in O$ such that $f(x)=(h\circ g)(x)$.
$$f(y)=||y||\cdot f(x)=||y||\cdot (h\circ g)(x)=h(||y||\cdot g(x))=(h\circ g)(y).$$
Therefore, if $f(x)$, $g(x)$ are in one element of the cover $\{Ot\ |\ t\in {\rm S}\}\in\mathcal U_{\rm S}$, then  $f(y)$, $g(y)$ are in one element of its extension to  ${\rm B}$. Since $f(0)=0$ for any $f\in U({\bf H})$, $\widetilde{\mathfrak{E}}_{\alpha}(\mathcal U_{\rm S})=\mathfrak{E}_{\alpha}(\mathcal U_{\rm S})$ (subbases of correspondent uniformities coincide).

The proof of the inclusion  $\mathfrak{E}_{\alpha}(\mathcal U_{\rm S})\supset R_{\Sigma_{\rm S}}$. In the proof of Theorem 1.1 from~\cite{U1998} (item (c)) it is shown that for any  $V_{L, \varepsilon}$, $L\in\mathcal L_{{\bf H}}$, $\varepsilon>0$, there exist $\delta>0$, $v\in V_{L, \varepsilon}$ and $u\in\st_L$ such that if $f, g\in U({\bf H})$ and 
$$|(f(x), y)-(g(x), y)|<\delta,\ \mbox{for all}\ x, y\in {\rm B}_L,\eqno{\rm (w.o.t)}$$
then $f=ugv$.

The condition (w.o.t) defines the basic entourage of the uniformity, completion with respect to which is the (compact) unit ball in the space of all bounded operators equipped with the {\it weak operator topology}. This uniformity coincides with the uniformity $\widetilde{\mathfrak{E}}_{\alpha}(\mathcal U_{\rm S})=\mathfrak{E}_{\alpha}(\mathcal U_{\rm S})$ by Lemma~\ref{unifunit}.

Let us note that if $f, g\in U({\bf H})$ and the condition (w.o.t) is fulfilled, then for the adjoint operators $f^*, g^*\in U({\bf H})$ the same condition (w.o.t) is fulfilled. Indeed, 
$$|(f^*(x), y)-(g^*(x), y)|=|(x, f(y))-(x, f(y))|=|(f(y), x)-(g(y), x)|,\ \mbox{and}$$
$$|(f(y), x)-(g(y), x)|<\delta \Longleftrightarrow\ |(f(x), y)-(g(x), y)|<\delta,\ \forall\ x, y\in L,\ ||x||\leq 1,\ ||y||\leq 1.$$

Therefore, if $f, g\in U({\bf H})$ and the condition (w.o.t) is fulfilled, then  $f^{-1}=f^*=u'g^*v'$, where $u\in\st_L$, $v\in V_{L, \varepsilon}$. Hence, $f=v^{-1}gu^{-1}$,  $u^{-1}\in\st_L$, $v^{-1}\in V_{L, \varepsilon}^{-1}=V_{L, \varepsilon}$ (since $||g^{-1}(x)-x||=||g(x)-x||$, $g\in U({\bf H})$) and, finally, $f\in V_{L, \varepsilon}g\st_L$. Hence,  $\mathfrak{E}_{\alpha}(\mathcal U_{\rm S})\supset R_{\Sigma_{\rm S}}$.

\medskip

(b) From (a) it follows that $b_r U({\bf H})$ is the closure of $\tilde\imath (U({\bf H}))$ in ${\rm B}^{\rm B}$. Since $||g(x)||=||x||$, $x\in {\rm B}$, for all $g\in U({\bf H})$,  $||f(x)||\leq ||x||$, $x\in {\rm B}$, for $f\in\cl~\tilde\imath (U({\bf H}))$ (if $||f(x)||>||x||$ for some  $x\in {\rm B}$, then for the nbd $Of(x)$ of $f(x)$ such that $Of(x)\cap\{y\in{\bf H}\ |\ ||y||\leq ||x||\}=\emptyset$, $[x, Of(x)]\cap \tilde\imath (U({\bf H}))=\emptyset$).

For $f\in\cl~\tilde\imath (U({\bf H}))$ let us show that $f$ is the restriction of a linear map on ${\bf H}$.

(1) If $||x||, ||y||\leq\frac12$, then $f(x+y)=f(x)+f(y)$. Indeed, $||x+y||\leq ||x||+||y||\leq 1$ and since  $||f(x)||\leq ||x||$, $x\in {\rm B}$, one has $||f(x+y)||\leq 1$, $||f(x)+f(y)||\leq 1$. Assume that $f(x+y)\ne f(x)+f(y)$. Take disjoint nbds $U$ of $f(x+y)$ and $V$ of $f(x)+f(y)$. Since addition in ${\bf H}$ is continuous there are nbds $Of(x)$ and $Of(y)$ such that $Of(x)+Of(y)\subset V$. Take the open set $O=[x, Of(x)]\cap [y, Of(y)]\cap [x+y, U]$ in ${\rm B}^{\rm B}$. $f\in O$, but $\tilde\imath (U({\bf H}))\cap O=\emptyset$. Hence, $f\not\in\cl~\tilde\imath (U({\bf H}))$.

(2) If $||x||\leq\frac12$, then $f(cx)=cf(x)$ for $c$ such that $|c|\leq 2$. Indeed, assume that  $f(cx)\ne cf(x)$, $||x||\leq\frac12$, $|c|\leq 2$, then $||cx||\leq 1$, $||cf(x)||\leq 1$. Take disjoint nbds $U$ of $f(cx)$ and $V$ of $cf(x)$. Since multiplication on the scalar in ${\bf H}$ is continuous there is nbd $Of(x)$ such that $cOf(x)\subset V$. Take the open set $O=[x, Of(x)]\cap [cx, U]$ in ${\rm B}^{\rm B}$. $f\in O$, but $\tilde\imath (U({\bf H}))\cap O=\emptyset$. Hence, $f\not\in\cl~\tilde\imath (U({\bf H}))$.

Put $\tilde f(x)=2||x|| f(\frac{x}{2||x||})$, $x\in{\bf H}$. For any $c$ and any $x\in{\bf H}$
$$\tilde f(cx)=2|c| ||x||f\big(\frac{cx}{2|c| ||x||}\big)\stackrel{\frac{c}{|c|}\leq 1,\ (2)}{=}2|c| ||x|| \frac{c}{|c|}f\big(\frac{x}{2||x||}\big)=2c ||x|| f\big(\frac{x}{2||x||}\big)=c\tilde f (x).\eqno{{\rm(C)}}$$
For any $x, y\in{\bf H}$
$$\tilde f(x+y)\stackrel{(C)}{=}2(||x||+||y||)f\big(\frac{x+y}{2(||x||+||y||)}\big)=2(||x||+||y||)\big(f\big(\frac{x}{2(||x||+||y||)}\big)+f\big(\frac{y}{2(||x||+||y||)}\big)\big)=$$
$$=2(||x||+||y||)f\big(\frac{x}{2(||x||+||y||)}\big)+2(||x||+||y||)f\big(\frac{y}{2(||x||+||y||)}\big)\stackrel{(C)}{=}\tilde f(x)+\tilde f(y).$$

From (2) and (C)  $\tilde f(x)=f(x)$, $x\in {\rm B}$. Therefore any $f\in\cl~\tilde\imath (U({\bf H}))$ is a restriction of a linear map on ${\bf H}$.

Let $\tilde f$ be any linear operator on  ${\bf H}$  of norm $\leq 1$, $f$ is its restriction to ${\rm B}$. The arbitrary nbd of $f$ is of the form $V=[x_1, Of(x_1)]\cap\ldots\cap [x_n, Of(x_n)]$, where,  without loss of generality, $x_1, \ldots, x_n\in {\bf H}$ are linearly independent, $Ox_1, \ldots, Ox_n$ are open subsets of  ${\rm B}$ in the weak topology such that ${\bf H}=L\oplus {\bf H}_1$ and $\pr^{-1}(\pr\ Of(x_k))=Of(x_k)$, $k=1,\ldots, n$, where $\pr$ is the orthogonal projection of ${\bf H}$ to the finite dimensional subspace $L$.

Let $\varphi$ be any linear isometry of the ${\rm span}(x_1, \dots, x_n)$ to the $n$-dimensional subspace of ${\bf H_1}$ and $\tilde\varphi$ its extension to the linear isometry of ${\bf H}$ (any linear isometry of the orthogonal complement to the ${\rm span}(x_1, \dots, x_n)$ to the  orthogonal complement to the image  $\varphi ({\rm span}(x_1, \dots, x_n))$ can be chosen). Then  $\tilde\varphi\in V$ and, hence, $f\in b_r U({\bf H})$.

(c) follows from (b) (all elements of $b_r U({\bf H})$ are continuous in the weak topology on ${\bf H}$).
\end{proof}

\end{document}